\documentclass[preprint,12pt]{elsarticle}

\usepackage{amssymb}
\usepackage{amsmath}
\usepackage{lineno}
\newcommand{\I}{\mathbf{I}}
\newcommand{\topt}{\mathrm{T}}
\newcommand{\subsett}{\subset}
\newcommand{\bsig}{\boldsymbol{\sigma}}
\newcommand{\btau}{\boldsymbol{\tau}}
\newcommand{\bphi}{\boldsymbol{\phi}}
\newcommand{\symg}{\nabla_s}
\newcommand{\ep}{\frac{1}{\varepsilon^{2}}}

\newcommand{\eun}{e_{u}^{n}}
\newcommand{\esn}{e_{\bsig}^{n}}

\newcommand{\tu }{\theta_{u}}
\newcommand{\ts }{\theta_{\bsig}}

\usepackage{booktabs}
\usepackage{mathtools}
\usepackage{stmaryrd}
\SetSymbolFont{stmry}{bold}{U}{stmry}{m}{n}
\usepackage{tikz}
\usepackage{graphicx}
\usepackage{subcaption}
\usepackage{hyperref}

\newtheorem{theorem}{Theorem}[section]

\newtheorem{lemma}[theorem]{Lemma}
\newtheorem{remark}[theorem]{Remark}

\newproof{proof}{Proof}
\newproof{proof2}{Proof of Lemma \ref{thminfsup}}

\numberwithin{equation}{section}
\numberwithin{table}{section}

\begin{document}

\begin{frontmatter}
\nolinenumbers
\title{A New Mixed Finite Element Method For The Cahn-Hilliard Equation}

\author[1,4]{Zhen Liu}
\ead{zliu37@pku.edu.cn} 
\affiliation[1]{organization={LMAM and School of Mathematical Sciences, Peking University},
                addressline={No.5 Yiheyuan Road,Haidian District}, 
                postcode={100871}, 
                city={Beijing}, 
                country={P. R. China}}
\affiliation[4]{organization={Chongqing Research Institute of Big Data, Peking University},
                addressline={Building 10, Science Valley, High-tech District}, 
                postcode={401332}, 
                city={Chongqing}, 
                country={P. R. China}}
\author[2]{Rui Ma  \corref{cor1}%
        \fnref{fn1}}
\ead{rui.ma@bit.edu.cn} 
\affiliation[2]{organization={Beijing Institute of Technology},
                addressline={No.5 South Zhong Guan Cun Street, Haidian District}, 
                postcode={100081}, 
                city={Beijing}, 
                country={P. R. China}}
\author[3]{Min Zhang \fnref{fn2}}
\ead{zhangmind01@bjfu.edu.cn} 
\affiliation[3]{organization={College of Science, Beijing Forestry University},
                addressline={No.35 Qinghua East Road, Haidian District}, 
                postcode={100083}, 
                city={Beijing}, 
                country={P. R. China}}

\cortext[cor1]{Corresponding author}
\fntext[fn1]{The work of the second author was partially supported by NSFC project 12301466.} 
\fntext[fn2]{{The work of the third author was partially supported by NSFC project 12401510 and BLX202347.}}

\begin{abstract}
This paper presents a new mixed finite element method for the Cahn-Hilliard equation. The well-posedness of the mixed formulation is established and the error estimates for its linearized fully discrete scheme are provided. {The new mixed finite element method provides a unified construction in two and three dimensions allowing for arbitrary polynomial degrees.} Numerical experiments are given to validate the efficiency and accuracy of the theoretical results.
\end{abstract}


\end{frontmatter}

\section{Introduction}
Let $\Omega$ be a bounded polyhedron in $\mathbb{R}^3$ or a polygon in $\mathbb{R}^2$ with the Lipschitz boundary $\partial \Omega$.
	This paper considers the Cahn-Hilliard equation 
	\begin{equation}
		\label{teq}
		\left\{  
		\begin{aligned}
			\frac{\partial u}{\partial t} -\Delta\left(-\Delta u+\frac{1}{\varepsilon^2} f(u)\right) &= g(\boldsymbol{x},t), && \text{in} ~ \Omega \left. \times(0, T\right], &\\
			{\partial_{\boldsymbol{n}}} u &= g_a(\boldsymbol{x},t), &&\text{on} ~ \partial\Omega \left. \times(0, T\right],& \\
			{\partial_{\boldsymbol{n}}} \left(\Delta u-\frac{1}{\varepsilon^2} f(u)\right) &= g_b(\boldsymbol{x},t),&&\text{on} ~ \partial\Omega \left. \times(0, T\right], &\\ u(\boldsymbol{x},0)&=u_0(\boldsymbol{x}),&&\text{in} ~ \Omega \times \{0\}.& 
		\end{aligned}  \right.
	\end{equation}
	Here the symbol $\Delta$ is the Laplacian operator, $T>0$ is a fixed constant, and ${\partial_{\boldsymbol{n}}} (\cdot)$ is the normal derivative where $\boldsymbol{n}$ is the unit outward normal vector of $\partial \Omega$. The function $f(u)$ is the derivative of a smooth chemical potential $F(u)$, and the widely used Ginzburg-Landau double-well potential $F (u) = \frac{1}{4} (u^{2} - 1)^{2}$ will be taken in this paper. The function $u_0(\boldsymbol{x})$ is the initial data. The functions $g(\boldsymbol{x}, t),g_a(\boldsymbol{x}, t),g_b(\boldsymbol{x}, t)$ serve as prescribed source terms assuming to be zero unless the contrary is explicitly stated. 
	
	The Cahn-Hilliard equation was introduced by Cahn  and Hilliard \cite{cahn1958free} to describe the complicated phase separation and coarsening phenomena in a solid where only two different concentration phases can exist stably. The unknown function $u(\boldsymbol{x},t)$ represents the concentration of each component and the parameter $\varepsilon > 0$ represents the inter-facial width. For more physical background and derivation of the Cahn-Hilliard equation, refer to \cite{cahn1965phase, Novick1998} and the references therein. The Cahn-Hilliard equation has been not only widely used in many complicated moving interface problems, multi-phase fluid flow and fluid dynamics \cite{Anderson1998, Shen2004complex, Feng2006pinching}, but also paired with other equations that describe physical behaviour of a given physical system, see, e.g., \cite{Abels2009twophase, feng2006NSCH, Grun2014twophase, Lee2002a}.
	
	Numerous studies on numerical methods have been conducted over the past thirty years, including finite element methods \cite{elliott1987numerical, wu2020analysis}, finite difference methods \cite{Sun1995FDM, Furihata2003FDM, Guo2016FEM}, finite volume methods \cite{Nabet2016FVM, Appadu2017FVM}, and spectral methods \cite{Feng2013spectral, Chen2013spectral}. 
	This paper focuses on finite element methods for the Cahn-Hilliard equation.
	Ensuring $C^1$ continuity is imperative when directly discretizing this fourth-order equation with conforming elements. Elliott and French \cite{elliott1987numerical} employed the spline finite element space to establish a fully discrete scheme for 1D problem. However, it is much more difficult to preserve $C^1$ continuity in high-dimensional spaces.
	To alleviate spatial continuity requirements, an alternative strategy is to use nonstandard finite element methods, such as nonconforming element methods \cite{elliott1989nonconforming, wu2020analysis, zhang2010nonconforming}, discontinuous Galerkin methods \cite{feng2007fully, feng2012analysis, kay2009discontinuous}, local discontinuous Galerkin methods \cite{Shu2007LDG}, weak Galerkin methods \cite{Zhang2019WG} and virtual element methods \cite{antonietti2016c}. 
	Besides, based on the Ciarlet-Raviart mixed method \cite{ciarlet1974mixed} for biharmonic equations, the chemical potential was introduced as an auxiliary variable in \cite{elliott1989second} which splits the Cahn-Hilliard equation into two coupled problems involving second-order spatial derivatives. This gives rise to a mixed finite element method that employs only continuous finite elements. Various fully discrete schemes using this mixed finite element method for spatial discretization have been developed and analyzed, see,  e.g., \cite{du1991numerical,elliott1992error,copetti1992numerical,diegel2016stability,chen2024recovery,feng2008posteriori}.
	
    In recent years, a mixed finite element method that introduces the Hessian as an auxiliary variable has been developed, see, e.g.,  \cite{pauly2020divdiv,chen2022finite,hu2021family,fuhrer2023mixed,chen2025new}.  
   Based on this method, this paper proposes a new mixed finite element method for the Cahn-Hilliard equation. 
   Compared to nonconforming methods, the new mixed method offers a unified construction of finite elements in two, three, and even higher dimensions \cite{chen2022finite1}, while also allowing for arbitrary polynomial degrees of discretization.
	The new mixed finite element method presented in this paper simultaneously seeks $u$ and the introduced variable $\bsig = \nabla^{2} u -\frac{1}{\varepsilon^{2}}f(u) \I $, which is sought in $H(\operatorname{divDiv},\Omega; \mathbb{S})$, consisting of symmetric matrix-valued $L^2$ functions whose $\operatorname{divDiv}$ belongs to $L^2(\Omega)$, with proper boundary conditions. The well-posedness of the mixed formulation is established by demonstrating its equivalence to the primal formulation of the Cahn-Hilliard equation. 
	Two distinct boundary conditions for the new mixed finite element method are explored. Specifically, one leads to an equivalent formula while the other is more sufficient. 
	This paper mainly considers the latter boundary condition and employs the  $H(\operatorname{divDiv},\Omega;\mathbb{S})$ conforming finite element in \cite{hu2021family} for the discretization of $\bsig$. Nevertheless, other $H(\operatorname{divDiv},\Omega;\mathbb{S})$ elements \cite{chen2022finite,fuhrer2023mixed} can also be employed.
	Besides, this paper provides detailed error estimates for the mixed finite element method.
	A key step in the proof is to utilize a broken $H^2$-norm of the numerical solution $u_h$ to control its $L^\infty$-norm. Mathematical induction is applied simultaneously to estimate the $L^\infty$-norm of $u_h$ and the $L^2$ errors of $\bsig_h$ and $u_{h}$.
	Numerical experiments are presented to validate the theoretical results.
    A postprocessing technique is employed to improve the convergence rates for $\boldsymbol{\sigma}_h$ and $u_h$.
	
	This paper is organized as follows. Section \ref{sec:preliminaries} introduces some notations and lemmas. Section \ref{sec:Mixedformulation} presents the mixed formulation and provides its equivalence to the primal formulation. Section \ref{sec4} gives the linearized fully discrete scheme of the mixed formulation and the error estimates.  Section \ref{sec:NumericalResults} provides numerical experiments to validate the theoretical results.

	
	\section{Preliminaries}
    \label{sec:preliminaries}
	Given a bounded domain $D \subsett \mathbb{R}^d$ with $d=2,3$, let $L^p(D;X)$ denote the standard Lebesgue spaces of functions within $D$, taking values in space $X$, with the corresponding norm $\| \cdot \|_{L^p(D)}$. 
	Similarly, let $H^m(D;X)$ denote the standard Sobolev spaces of functions for positive integers $m$ with the corresponding norm $\| \cdot \|_{H^m(D)}$,
	and let $C^m(D ; X)$ denote the space of $m$-times continuously differentiable functions.
	The $L^2$-scalar product over $D$ is denoted as $(\cdot, \cdot)_D$.
	The subscript $D$ in $\| \cdot \|_{L^p(D)}$, $\| \cdot \|_{H^m(D)}$ and $(\cdot, \cdot)_D$ will be omitted if $D=\Omega$.
	In this paper, $X$ could be $\mathbb{R}$, $\mathbb{R}^d$ or $\mathbb{S}$, where $\mathbb{S}$ denotes the set of symmetric $\mathbb{R}^{d \times d}$ matrices. 
	If $X=\mathbb{R}$, then $L^2(D)$ abbreviates $L^2(D;X)$, similarly for $H^m(D)$ and $C^m(D)$. 
	Let $H^{-m}(D)$ denote the dual space of $H_0^m(D)$ and $\langle \cdot, \cdot \rangle_{ H^{-m}\times H_0^m}$ denote the duality product on $H^{-m}(D) \times H_0^{m}(D)$.
	Let $Y$ be a real Banach space with the norm $\|\cdot\|_{Y}$. 
	The space $L^\infty(0, T ; Y)$ consists of all measurable functions $u:[0, T]$ $\rightarrow Y$ with
	\begin{equation}
		\label{deflinfty}
		\|u\|_{L^{\infty}(0, T ; Y)}\coloneqq\sup _{0 \leq t \leq T}\|u(t)\|_{Y}<\infty .
	\end{equation}
	
	If $D \subset \mathbb{R}^3$ is a polyhedron, then let $\mathcal{F}(D), \mathcal{E}(D)$ and $\mathcal{V}(D)$ be the sets of all faces, edges and vertices of $D$, respectively. For a face $F \in \mathcal{F}(D)$, the unit outer normal vector $\boldsymbol{n}_F$ is fixed. For any $F \in \mathcal{F}(D)$, let $\mathcal{E}(F)$ be the set of all edges of $F$. {Specifically, for each $e \in \mathcal{E}(F)$, denote by $\boldsymbol{n}_{F, e}$ the unit vector parallel to $F$ and outward normal to $\partial F$.} Given an edge $e \in \mathcal{E}(D)$, the unit tangential vector $\boldsymbol{t}_e$, and two unit normal vectors, $\boldsymbol{n}_{e, 1}$ and $\boldsymbol{n}_{e, 2}$, are fixed. If $D \subset \mathbb{R}^2$ is a polygon, then let $\mathcal{E}(D)$ and $\mathcal{V}(D)$ be the sets of all edges and vertices of $D$, respectively. Given an edge $e \in \mathcal{E}(D)$, the unique unit outer normal vector $\boldsymbol{n}_e$ and the unit tangential vector $\boldsymbol{t}_e$ are fixed. In the absence of ambiguity, the symbol $\boldsymbol{n}$ instead of $\boldsymbol{n}_e$ and $\boldsymbol{n}_{F}$ in two and three dimensions, respectively, will be used to denote the unit outer normal vector of $\partial D$. Let $\I$ denote the identity matrix in $\mathbb{R}^{d}$. 
	 For $D \subset \mathbb{R}^3$, given $F \in \mathcal{F}(D)$, {vector} $\boldsymbol{v}$ and tensor $\boldsymbol{\tau}$, {define}
	\begin{equation}
		\label{defoperator}
		\begin{aligned}
			\Pi_F \boldsymbol{v} & \coloneqq (\I-\boldsymbol{n}_{F}\boldsymbol{n}_F^\topt)\boldsymbol{v}, \\
			\operatorname{div}_F(\btau\boldsymbol{n}_F)& \coloneqq (\boldsymbol{n}_F \times \nabla) \cdot(\boldsymbol{n}_F \times(\btau\boldsymbol{n}_F)).
		\end{aligned}
	\end{equation}
 
	In this paper, the differential operator $\operatorname{Div}$ for matrix functions is applied row-wise. The following lemmas are presented for later use.
	
	\begin{lemma}[Green's identity in 3D \cite{chen2022finite,Niemi2019}]
		\label{Green}
		Suppose $\Omega \subset \mathbb{R}^3$. Let $\btau \in H(\operatorname{divDiv}, \Omega ; \mathbb{S}) \cap C^1(\bar \Omega; \mathbb{S}) $ and $v \in H^2(\Omega)$. Then the following equality holds
		\begin{equation}
			\label{Green1}
			\begin{aligned}(\operatorname{divDiv} \btau, v) =\left(\btau, \nabla^2 v\right) -\sum_{F \in \mathcal{F}(\Omega)}(\btau \boldsymbol{n}_F, \nabla v)_F+\sum_{F \in \mathcal{F}(\Omega)}\left(\boldsymbol{n}_{F}^{\mathrm{T}} \operatorname{Div} \btau, v\right)_F,\end{aligned}
		\end{equation}
		and furthermore
		\begin{equation}
			\label{Green2}
			\begin{aligned}
				(\operatorname{divDiv}\btau, & v)  =\left(\btau, \nabla^2 v\right) -\sum_{F \in \mathcal{F}(\Omega)} \sum_{e \in \mathcal{E}(F)}\left(\boldsymbol{n}_{F, e}^{\mathrm{T}} \btau \boldsymbol{n}_F, v\right)_e \\
				& -\sum_{F \in \mathcal{F}(\Omega)}\left[\left(\boldsymbol{n}_F^{\mathrm{T}} \btau \boldsymbol{n}_F, \frac{\partial v}{\partial \boldsymbol{n}_F}\right)_F-\left(2 \operatorname{div}_F\left(\btau \boldsymbol{n}_F\right)+\frac{\partial\left(\boldsymbol{n}_F^{\mathrm{T}} \btau \boldsymbol{n}_F\right)}{\partial \boldsymbol{n}_F}, v \right)_F\right].
			\end{aligned}
		\end{equation}
	\end{lemma}

	\begin{lemma}[Discrete Gronwall's inequality \cite{Heywood1990}]
		\label{lemgron}
		Let $\tau, B$ and $a_k, b_k, c_k, \gamma_k$, for integers $k \geq 0$, be non-negative numbers such that
		$$
		a_n+\tau \sum_{k=0}^n b_k \leq \tau \sum_{k=0}^n \gamma_k a_k+\tau \sum_{k=0}^n c_k+B, \text { for } n \geq 0,
		$$
		suppose that $\tau \gamma_k<1$, for all $k$, and set $\sigma_k=\left(1-\tau \gamma_k\right)^{-1}$. Then
		$$
		a_n+\tau \sum_{k=0}^n b_k \leq \exp \left(\tau \sum_{k=0}^n \gamma_k \sigma_k\right)\left(\tau \sum_{k=0}^n c_k+B\right), \text { for } n \geq 0.
		$$
	\end{lemma}
	

	\section{Mixed variational formulation}
    \label{sec:Mixedformulation}
	This section presents the new mixed variational formulation for the Cahn-Hilliard equation \eqref{teq} and demonstrates its equivalence to the primal variational formulation.
	
	Define
	$$H_E^2(\Omega) \coloneqq \{  v \in H^2(\Omega): {\partial_{\boldsymbol{n}}} v = 0, \text{on}~ \partial \Omega \} .$$  
	The primal formulation of \eqref{teq} reads: Find $u \in H^2_{E}(\Omega)$ such that
	\begin{equation}
		\label{classicalvar}
		(\frac{\partial u}{\partial t}, v ) + ( \nabla^2 u , \nabla^2 v )= ( \ep f(u) \I, \nabla^2 v), \quad \forall v \in H_E^2(\Omega).
	\end{equation}
    According to {\cite{feng2001numerical}}, \eqref{classicalvar} possesses a unique solution for $f(u)=u^3-u$ when $\partial \Omega$ is smooth enough.
	
	Introducing the auxiliary variable $$\bsig  \coloneqq \nabla^{2} u -\frac{1}{\varepsilon^{2}}f(u) \I,$$
	one can reformulate the Cahn-Hilliard equation \eqref{teq} into
	\begin{equation}
		\label{meq}
		\left\{  
		\begin{aligned}
			\frac{\partial u}{\partial t} + \operatorname{divDiv} \bsig & = 0,\quad && \text{in}~\Omega \left. \times(0, T \right],  \\
			\bsig &= \nabla^{2} u -\frac{1}{\varepsilon^{2}}f(u)\I, \quad &&\text{in}~ \left. \Omega \times(0, T \right],  \\
			{\partial_{\boldsymbol{n}}} u & = 0,\quad &&\text{on} ~\partial \Omega \left. \times(0, T\right], \\
			\operatorname{Div} \bsig \cdot \boldsymbol{n} &= 0,\quad &&\text{on} ~\partial\Omega \left. \times(0, T\right], \\
            u(\boldsymbol{x},0)&=u_0(\boldsymbol{x}),&&\text{on} ~ {\Omega}  \times \{0\}.& \\
		\end{aligned} \right.
	\end{equation}
	Define
	$$
	H(\operatorname{divDiv}, \Omega ; \mathbb{S}) \coloneqq \left\{\btau \in L^2(\Omega ; \mathbb{S}): \operatorname{divDiv} \btau \in L^2(\Omega)\right\},
	$$
	equipped with the squared norm
	$
	\| \btau  \|_{H(\operatorname{divDiv})}^2 \coloneqq \|\btau\|_{L^2}^2+\|\operatorname{divDiv} \btau\|_{L^2}^2 .
	$
	By employing a similar approach as presented in  \cite{Niemi2019,Walter2018}, define a nonstandard Sobolev space with a portion of the trace vanishing as 
	\begin{equation}
		\label{Sigma}
		\Sigma \coloneqq \left\{ \btau \in H(\operatorname{divDiv}, \Omega ; \mathbb{S}): (\operatorname{divDiv}  \btau, v) = (\btau, \nabla^2v ),\forall v \in H_E^2(\Omega)\right\}.
	\end{equation}
	Let $V \coloneqq L^2(\Omega)$. Then the new mixed variational formulation reads: Find $(\bsig, u) \in \Sigma \times L^6(\Omega)$ such that
	\begin{equation}
		\label{mvar2}
		\left\{  
		\begin{aligned}
			(\frac{\partial u}{\partial t},v) + (\operatorname{divDiv} \bsig, v) &= 0,\quad& &\forall v \in V,&\\
			(\bsig, \btau)-(\operatorname{divDiv}\btau,u)&= (-\frac{1}{\varepsilon^{2}}f(u)\I,\btau),\quad&& \forall\btau \in \Sigma. &
		\end{aligned}    \right.
	\end{equation}
	
	\begin{remark}
		The requirement $u \in L^6(\Omega)$ is reasonable since $f(u) = u^3-u \in L^2(\Omega)$ requires at least $u \in L^6(\Omega)$. Indeed, the exact solution $u$ of \eqref{classicalvar} belongs to $H^2(\Omega)$, implying $u \in L^6(\Omega)$.
	\end{remark}
	
	The following two lemmas give the boundary conditions for smooth functions in $\Sigma$. Define the jump term for $F_1, F_2 \in \mathcal{F}(\Omega)$ and $F_1\cap F_2 = e$ as
	$$ \llbracket \boldsymbol{n}_{F, e}^{\mathrm{T}} \btau \boldsymbol{n}_F\rrbracket_{e} \coloneqq  \boldsymbol{n}_{{F_1}, e}^{\mathrm{T}} \btau \boldsymbol{n}_{F_1} + \boldsymbol{n}_{{F_2}, e}^{\mathrm{T}} \btau \boldsymbol{n}_{F_2}.
	$$
	
	\begin{lemma}
		\label{3dboundary}
		Suppose $\Omega \subset \mathbb{R}^3$. 
		If $\btau \in  H(\operatorname{divDiv}, \Omega ; \mathbb{S}) \cap C^1(\bar \Omega; \mathbb{S})$, then $\btau \in \Sigma$ if and only if 
		\begin{equation}
			\label{3Dboundary}
			\begin{aligned}
				\llbracket \boldsymbol{n}_{F, e}^{\mathrm{T}} \btau \boldsymbol{n}_F\rrbracket_{e} = 0,\quad &\forall e\in \mathcal{E}(\Omega), \\
				2\operatorname{div}_{F}(\btau\boldsymbol{n}_{F})+ \frac{\partial\left(\boldsymbol{n}_F^{\mathrm{T}} \btau \boldsymbol{n}_F\right)}{\partial \boldsymbol{n}_F} =0, \quad &\forall F \in \mathcal{F}(\Omega).
			\end{aligned}
		\end{equation}
	\end{lemma}
	
	\begin{proof}
		For any $\btau \in  H(\operatorname{divDiv}, \Omega ; \mathbb{S}) \cap C^1(\bar \Omega; \mathbb{S})$, Green's identity \eqref{Green2} shows
		\begin{equation}
			\begin{aligned}
				\label{Green10}
				(\operatorname{divDiv}\btau, v) & =\left(\btau, \nabla^2 v\right) -\sum_{F \in \mathcal{F}(\Omega)} \sum_{e \in \mathcal{E}(F)}\left(\boldsymbol{n}_{F, e}^{\mathrm{T}} \btau \boldsymbol{n}_F, v\right)_e \\
				& \quad + \sum_{F \in \mathcal{F}(\Omega) }  \left(2 \operatorname{div}_F\left(\btau \boldsymbol{n}_F\right)+\frac{\partial\left(\boldsymbol{n}_F^{\mathrm{T}} \btau \boldsymbol{n}_F\right)}{\partial \boldsymbol{n}_F}, v \right)_F,
			\end{aligned}
		\end{equation}
		for all $v\in H^2_E(\Omega)$.
		If $\btau \in \Sigma$, then the arbitrariness of $v$ and \eqref{Sigma} show that the boundary terms in \eqref{Green10} equal zero, which implies \eqref{3Dboundary}. On the contrary, if \eqref{3Dboundary} holds, this and \eqref{Green10} lead to $\boldsymbol{\tau} \in \Sigma$. This concludes the proof.
	\end{proof}
	
	Define the jump term for $e_1,e_2 \in \mathcal{E}(\Omega)$ and $e_1 \cap e_2 = \boldsymbol{x}$ as
	$$ 
	\llbracket \boldsymbol{t}^{T} \btau \boldsymbol{n} \rrbracket_{\boldsymbol{x}} \coloneqq \operatorname{sign}_{e_{1}, \boldsymbol{x}} \boldsymbol{t}_{e_1}^{T} \btau \boldsymbol{n}_{e_1} + \operatorname{sign}_{e_{2}, \boldsymbol{x}}\boldsymbol{t}_{e_2}^{T} \btau \boldsymbol{n}_{e_2}
	$$
	with
	$$
	\mathrm{sign}_{e,\boldsymbol{x}}\coloneqq
	\left\{
	\begin{aligned}
		1,&\mathrm{~if~} \boldsymbol{x}\text{ is the end point of }e,&\\
		-1,&\mathrm{~if~} \boldsymbol{x}\text{ is the start point of }e.&
	\end{aligned}
	\right.
	$$
	\begin{lemma}
		\label{2dboundary}
	Suppose $\Omega \subset \mathbb{R}^2$. 
	If $\btau \in  H(\operatorname{divDiv}, \Omega ; \mathbb{S}) \cap C^1(\bar \Omega; \mathbb{S})$, then $\btau \in \Sigma$ 
	if and only if 
		\begin{equation}
			\label{H22D}
			\llbracket \boldsymbol{t}^{\mathrm{T}} \btau \boldsymbol{n} \rrbracket_{\boldsymbol{x}} = 0, \, \forall \boldsymbol{x} \in \mathcal{V}(\Omega),\quad 
			2 \frac{\partial\left(\boldsymbol{t}_e^{\mathrm{T}} \btau \boldsymbol{n}_e\right)}{\partial \boldsymbol{t}_e}  + \frac{\partial\left(\boldsymbol{n}_e^{\mathrm{T}} \btau \boldsymbol{n}_e\right)}{\partial \boldsymbol{n}_e}=0, \, \forall e \in \mathcal{E}(\Omega).
		\end{equation}
	\end{lemma}
	\begin{proof}
	The proof is similar to that of Lemma \ref{3dboundary} by using Green's identity in two dimensions \cite[Lemma 4.2]{chen2022finite}.
	\end{proof}
	
	The subsequent theorem shows the equivalence of the mixed formulation \eqref{mvar2} and the primal one \eqref{classicalvar}.
	\begin{theorem}
		The problems \eqref{classicalvar} and \eqref{mvar2} are fully equivalent, i.e., if $u \in H_E^2(\Omega)$ solves \eqref{classicalvar}, then $\bsig=\nabla^2 u -\ep f(u)\I \in \Sigma$ and $(\bsig,u)$ solves \eqref{mvar2}. And, vice versa, if $(\bsig,u) \in \Sigma \times V$ solves \eqref{mvar2}, then $u \in H_E^2(\Omega)$ and $u$ solves \eqref{classicalvar}.
	\end{theorem}
	\begin{proof}
		Suppose that $u \in H_E^2(\Omega)$ solves \eqref{classicalvar}. Then $\bsig= \nabla^2 u -\ep f(u)\I \in L^2(\Omega; \mathbb{S})$ and
		\begin{equation}
			\label{divDivUtt}
			( \frac{\partial u}{\partial t} , v ) +  ( \bsig, \nabla^2 v ) = 0, \quad \forall v \in H_E^2(\Omega).
		\end{equation}
		This and $C_0^{\infty}(\Omega) \subset H^2_E(\Omega)$ lead to $\operatorname{divDiv} \bsig = -  \frac{\partial u}{\partial t} \in L^2(\Omega)$. The first row in \eqref{mvar2} immediately follows. This and \eqref{divDivUtt} yield
		$$(\bsig,\nabla^2 v) = (-  \frac{\partial u}{\partial t} , v ) = (\operatorname{divDiv} \bsig, v), \quad \forall v \in H^2_E(\Omega).$$
		This shows $\bsig \in \Sigma$. The definition of $\Sigma$ and $u \in H_E^2(\Omega)$ imply
		$$
		( \operatorname{divDiv} \btau, u )= ( \btau, \nabla^2 u) = ( \btau, \bsig ) {+}  (\ep f(u) \I,  \btau), \quad \forall \btau \in \Sigma,
		$$
		which proves the second row in \eqref{mvar2}. 
		
		Suppose that $(\bsig,u) \in \Sigma \times L^6(\Omega)$ solves \eqref{mvar2}. Since $\bsig+\frac{1}{\varepsilon^{2}}f(u)\I \in L^2(\Omega; \mathbb{S})$ and $C_0^{\infty}(\Omega; \mathbb{S}) \subset \Sigma$, it follows from the second row of \eqref{mvar2} that $\nabla^2 u = \bsig+\frac{1}{\varepsilon^{2}}f(u)\I \in L^2(\Omega;\mathbb{S})$. This and Ne\v{c}as's inequality {\cite[Theorem IV.1.1.]{boyer2012mathematical}}
        $$
        \|\nabla u\|_{L^2} \leq C\left(\| u\|_{L^2 }+\|\nabla^2 u\|_{H^{-1}}\right)
        $$
        imply $u \in H^2(\Omega)$. The choice of $v \in H^2_E(\Omega)$ in the fist row of \eqref{mvar2} derives that $u$ solves \eqref{classicalvar} since
		$$
		\begin{aligned}
			(\frac{\partial u}{\partial t}, v) &= - (\operatorname{divDiv} \bsig, v) =-(\bsig, \nabla^2 v) \\
			&= -(\nabla^2 u, \nabla^2 v) + (\frac{1}{\varepsilon^2}f(u)\I,v), \quad \forall v\in H^2_E(\Omega).
		\end{aligned}
		$$
		It remains to show $u \in H^2_E(\Omega)$. 
		The second row in \eqref{mvar2} implies
		\begin{equation*}
		(\nabla^2 u, \btau) = (\operatorname{divDiv}\btau,u)
		\end{equation*}
		for any $\btau \in C^{1}(\bar{\Omega};\mathbb{S}) \cap \Sigma$. 
		This, Green's identity in Lemma \ref{Green}, and Lemma \ref{3dboundary} show ${\partial_{\boldsymbol{n}}} u =0 $ for $\Omega \subsett \mathbb{R}^3$. Similarly, Green's identity in \cite[Lemma 4.2]{chen2022finite} and Lemma \ref{2dboundary} show ${\partial_{\boldsymbol{n}}} u =0 $ for $\Omega \subsett \mathbb{R}^2$.
		This concludes the proof.
	\end{proof}
	
	Recall $\Pi_{F}$ from \eqref{defoperator}. This paper will use the following more sufficient boundary conditions compared with those in \eqref{3Dboundary} and \eqref{H22D} for the mixed finite element method.
	
	\begin{lemma}
		\label{sufficientbd}
		Suppose $\Omega \subset \mathbb{R}^3$.  If $\btau \in  H(\operatorname{divDiv}, \Omega ; \mathbb{S}) \cap C^1(\bar \Omega; \mathbb{S})$ satisfies
		\begin{equation}
			\label{3dboundarysuffi}
			\Pi_{F} \left( \btau\boldsymbol{n}_{F} \right)=0,~ \boldsymbol{n}_{F}^{\topt}\operatorname{Div} \btau=0, ~\forall F \in \mathcal{F} (\Omega),
		\end{equation}
		then $\btau \in \Sigma$. Similarly, for $\Omega \subset \mathbb{R}^2$, if $\btau \in  H(\operatorname{divDiv}, \Omega ; \mathbb{S}) \cap C^1(\bar \Omega; \mathbb{S})$ satisfies 
		$\boldsymbol{t}_e^{\mathrm{T}} \btau \boldsymbol{n}_e = 0$ {and} $\boldsymbol{n}_e^{\mathrm{T}}\operatorname{Div} \btau  = 0$ for all $e\in \mathcal{E}(\Omega)
		$, then $\btau \in \Sigma$.
	\end{lemma}
	
	\begin{proof}
		Substituting \eqref{3dboundarysuffi} into Green's identity \eqref{Green1} proves $\btau \in \Sigma$ for $\Omega \subset \mathbb{R}^3$. Similar arguments suit for $\Omega \subset \mathbb{R}^2$.
	\end{proof}

	\begin{remark}
        {The conditions in Lemma \ref{sufficientbd} are consistent
        with the boundary conditions of the exact solution $(\bsig,u)$ of the mixed formulation \eqref{mvar2}}.
	\end{remark}

	
	\section{Mixed finite element method}
	\label{sec4}
	
	This section establishes the linearized fully discrete scheme of \eqref{mvar2} and presents the error estimates.
	
	\subsection{Linearized fully discrete scheme}
	
	Let $N$ be a positive integer and let $\{t^n\}_{n=0}^{N}$ constitute a uniform partition of $[0, T]$ with the step size $\tau= \frac{T}{N}$. Select the backward Euler method as the temporal discretization method and incorporate the nonlinear term in the last time step. Then, the first-order semi-implicit method reads: Find $(\hat{\bsig}^n, \hat{u}^n) \in \Sigma \times L^6(\Omega)$ such that for $n=1,2,\cdots,N$,
	\begin{equation*}
		\left\{   
		\begin{aligned}
			\frac{1}{\tau}(\hat{u}^{n}-\hat{u}^{n-1},v) + (\operatorname{divDiv} \hat{\bsig}^{n}, v)&= 0,  &&\forall  v \in V,& \\
			(\hat{\bsig}^{n}, \btau) -(\operatorname{divDiv}\btau,\hat{u}^{n}) &= (-\frac{1}{\varepsilon^{2}}f(\hat{u}^{n-1})\I,\btau), && \forall \btau \in  \Sigma.&
		\end{aligned}\right.
	\end{equation*}
	At the initial time step, let $\hat{u}^0(\boldsymbol{x},0)=u_0(\boldsymbol{x})$.
	
	Suppose that $\mathcal{T}_h$ is a shape regular and quasi-uniform subdivision of $\Omega$ consisting of triangles in two dimensions and tetrahedrons in three dimensions. Define $h$ as the maximum of the diameters of all the elements $K \in \mathcal{T}_h$. 
	The jump of $u$ across an interior $d-1$ face $G$ shared by neighboring elements $K_{+}$ and $K_{-}$ is defined by
	$[u]_G \coloneqq ( \left.u \right|_{K_{+}}-\left.u\right|_{K_{-}} )\left. \right|_{G}.$
	When it comes to any boundary face $G \subset \partial \Omega$, the jump $[\cdot]_G$ reduces to the trace. 
	For $\Omega \subset \mathbb{R}^3$, let $\mathcal{V}_h, \mathcal{E}_h, \mathcal{F}_h, \mathcal{V}_h^{i}, \mathcal{E}_h^{i}, \mathcal{F}_h^{i}$ denote the set of all the vertices, the edges, the faces, the interior vertices, the interior edges, the interior faces of $\mathcal{T}_h$, respectively.
	Let $h_e$, $h_F$ and $h_K$ denote the diameters of $e \in \mathcal{E}_h$, $F\in \mathcal{F}_h$ and $K \in \mathcal{T}_h$, respectively. For $\Omega \subset \mathbb{R}^2$, the sets $\mathcal{F}_h$ and $ \mathcal{F}_h^i$ are empty and $h_F$ has no meaning.
	
	To give the discrete spaces, introduce the $H(\operatorname{div}, \Omega;\mathbb{S})$ conforming spaces in \cite{HuZhang2014,HuZhang2015,Hu2015JCM} with $k \geq 3$ as follows
	\begin{equation*}
		\mathcal{U}_{\partial K, b} \coloneqq \left\{  \btau \in P_{k}(K; \mathbb{S}):  \btau \boldsymbol{n} =  0, \text{ on }  \partial K\right\}, 
	\end{equation*}
	\begin{equation*}
		\begin{aligned}
			\mathcal{U}_{k,h}\coloneqq\big \{ &\btau \in H(\operatorname{div},\Omega;\mathbb{S}):\btau=\btau_c+\btau_b, \btau_c\in H^1(\Omega;\mathbb{S}),  \\
			&	\btau_{c}|_{K}\in P_{k}(K;\mathbb{S}) , \btau_{b}|_{K}\in \mathcal{U}_{\partial K,b},\forall K\in\mathcal{T}_{h} \big \}.
		\end{aligned}
	\end{equation*}
	{Introduce} the $H(\operatorname{div}, \Omega;\mathbb{S}) \cap H(\operatorname{divDiv}, \Omega;\mathbb{S})$ conforming spaces in \cite{hu2021family} with zero boundary conditions in $\mathbb{R}^3$ as 
	$$
	\begin{aligned}
		\Sigma_{h}\coloneqq\{ \btau_h  \in  \mathcal{U}_{k,h}:\Pi_{F}(\btau_h \boldsymbol{n}_F) = 0, \, \forall F \in \mathcal{F}_h \backslash \mathcal{F}_{h}^i, \,  [\operatorname{Div}\btau_h \cdot \boldsymbol{n}_F]_F = 0,\, \forall F \in \mathcal{F}_h
		\},
	\end{aligned}
	$$
	and in $\mathbb{R}^2$ as
	$$
	\begin{aligned}
		\Sigma_{h}\coloneqq\{ \btau_h  \in  \mathcal{U}_{k,h}:\left. \boldsymbol{t}_e^{\topt}\btau_h \boldsymbol{n}_e \right|_e=0, \, \forall e \in \mathcal{E}_h \backslash \mathcal{E}_h^i,\,  [\operatorname{Div}\btau_h \cdot \boldsymbol{n}_e]_e = 0, \, \forall e \in \mathcal{E}_h
		\}.
	\end{aligned}
	$$
	Define the piecewise polynomial spaces
	$$
	V_{h}\coloneqq\left\{v_h \in L^2(\Omega):\left. v_h \right|_K \in P_{k-2}(K),\, \forall K \in \mathcal{T}_h\right\},
	$$
	equipped with the mesh-dependent semi-norm in $\mathbb{R}^3$ as
	$$
	|v_h|_{2, h}^2\coloneqq\sum_{K \in \mathcal{T}_h}|v_h|_{H^2(K)}^2+ \sum_{F \in \mathcal{F}_h^i}  h_F^{-3}\|[v_h]_F\|_{L^2(F)}^2+ \sum_{F \in \mathcal{F}_h} h_F^{-1}\left\|[{\partial_{\boldsymbol{n}}} v_h]_F\right\|_{L^2(F)}^2,
	$$
	and in $\mathbb{R}^2$ as
	$$
	|v_h|_{2, h}^2\coloneqq\sum_{K \in \mathcal{T}_h}|v_h|_{H^2(K)}^2+\sum_{e \in \mathcal{E}_h^i} h_e^{-3}\|[v_h]_e\|_{L^2(e)}^2+ \sum_{e \in \mathcal{E}_h} h_e^{-1}\left\|[{\partial_{\boldsymbol{n}}} v_h]_e\right\|_{L^2(e)}^2.
	$$
	The following lemma holds for the mesh-dependent semi-norm which will be used in the error estimates.
	{\begin{lemma} 
			\label{leminfsup}
			There exists some constant $\beta>0$ such that the following inf-sup condition holds
			\begin{equation}
				\label{infsup0}
				\sup_{\btau_{h}\in \Sigma_{h}} \frac{ (\operatorname{divDiv} \btau_{h}, v_h)}{\| \btau_h \|_{L^2}} \geq \beta |v_h|_{2,h},\quad \forall v_h \in V_h.
			\end{equation}
		\end{lemma}
		\begin{proof}
			Based on \cite[Lemma 3.4]{hu2021family}, one can construct a $\btau_{h} \in \Sigma_{h}$ on each $K \in \mathcal{T}_h$ in $\mathbb{R}^3$ with modifications
			\begin{equation*}
				\begin{aligned}
					(\boldsymbol{n}_F^{\topt}\btau_h \boldsymbol{n}_F, q)_{F} &= (h_F^{-1}[{\partial_{\boldsymbol{n}}} v_h]_F, q)_F, &&\forall q \in P_{k-3}(F), F\in \mathcal{F}(K),&\\
					(\boldsymbol{t}_{F,i}^{\topt}\btau_h \boldsymbol{n}_F, q)_F &= 0, &&\forall q \in P_{k-3}(F), F\in \mathcal{F}(K),&\\
					(\operatorname{Div} \btau_h \cdot \boldsymbol{n}_F, q)_{F} &= -(h_F^{-3} [v_h]_F,q)_{F}, &&\forall q \in P_{k-1}(F), F\in \mathcal{F}(K)\cap \mathcal{F}_h^i,&\\
					(\operatorname{Div} \btau_h \cdot \boldsymbol{n}_F, q)_{F} &= 0, &&\forall q \in P_{k-1}(F), F\in \mathcal{F}(K) \cap (\mathcal{F}_h \backslash \mathcal{F}_h^i).&\\
				\end{aligned}
			\end{equation*}
			These modifications lead to 
            \begin{equation}
            \label{infsup2}
           (\operatorname{divDiv} \btau_h, v_h) = |v_h|_{2,h}^2.
            \end{equation}
            The scaling argument leads to 
            $\| \btau_{h} \|_{L^2} \leq C |v_h|_{2,h}$ 
            with a positive constant $C$. This and \eqref{infsup2} prove \eqref{infsup0} in $\mathbb{R}^3$ with $\beta = \frac{1}{C}$. 
			Similar arguments prove \eqref{infsup0} in $\mathbb{R}^2$.
	\end{proof}}
	
	Then, given the solution at time $t^{n-1}$, the fully discrete scheme seeks $(\bsig_{h}^n, u_{h}^n) \in (\Sigma_h, V_h)$ such that for $n=1,2,\cdots,N$,
	\begin{equation}
		\label{ds}
		\left\{  
		\begin{aligned}
			\frac{1}{\tau}(u_{h}^{n}-u_{h}^{n-1},v_{h}) + (\operatorname{divDiv} \bsig_{h}^{n}, v_{h})&= 0, &&\forall  v_{h} \in {V_h},& \\
			(\bsig_{h}^{n}, \btau_{h} )- (\operatorname{divDiv}\btau_{h},u_{h}^{n})&= (-\frac{1}{\varepsilon^{2}}f(u_{h}^{n-1})\I,\btau_{h}),&& \forall \btau_{h} \in {\Sigma_h}.&
		\end{aligned}   \right.
	\end{equation}
	At the initial time step, let
	$$u_{h}^{0}=\Pi_{h}u_{0}, \quad \bsig_{h}^{0}=\Pi_h \bsig_{0} = \Pi_h \left(  \nabla^2 u_0 - \frac{1}{\varepsilon^2}f(u_0)\I \right),$$
	where the projection operator $\Pi_h$ will be defined in \eqref{proj} below.
	
	\begin{remark}
		Note that the exact solution $u$ satisfies mass conservation
		\begin{equation*}
			\frac{\partial }{\partial t}\int_{\Omega} u \mathrm{d}x = \int_{\Omega} \Delta\left(-\Delta u+\frac{1}{\varepsilon^2} f(u)\right) \mathrm{d} x = \int_{\partial \Omega} {\partial_{\boldsymbol{n}}} \left(-\Delta u+\frac{1}{\varepsilon^2} f(u)\right) \mathrm{d} x = 0.
		\end{equation*}
		By taking $v_h = 1$ in \eqref{ds}, it can be observed from the first row that $\int_{\Omega} u_{h}^n \mathrm{d}x = \int_{\Omega} u_{h}^{n-1} \mathrm{d}x$.
	\end{remark}
	
	Define the space
	$$
	{V_h^0} \coloneqq \left\{ v_h \in V_h : \int_{\Omega} v_h \mathrm{d} x = 0 \right\}.
	$$
	Note that the semi-norm $| \cdot |_{2,h}$ of $V_h$ is a norm of $V_h^0$. To give the error estimates, introduce the projection operator $\Pi_h:\left(\Sigma, V\right) \rightarrow$ $\left(\Sigma_h, {V_h} \right)$, for given exact solution $(\bsig,u)$ of \eqref{mvar2} at any time, such that $\left(\Pi_h \bsig ,\Pi_h u\right) \in  \left(\Sigma_h, {V_h} \right)$ with $\int_{\Omega} \Pi_h u \mathrm{d} x = \int_{\Omega} u \mathrm{d} x$ satisfies
	\begin{equation}
		\label{proj} \left\{ 
		\begin{aligned}
			\left( \operatorname{divDiv} \left(\Pi_h \bsig -\bsig\right), v_h\right)&=0, && \forall v_h \in {V_h^0},&\\
			\left(\Pi_h \bsig -\bsig, {\btau_h} \right) - \left(\operatorname{divDiv} {\btau_h},\Pi_h u -u\right)&=0, &&\forall {\btau_h} \in {\Sigma_h}.& 
		\end{aligned}\right.
	\end{equation}
	
	Denote the projection error functions as
	$$
	\theta_{\bsig}\coloneqq \bsig -\Pi_h \bsig, \quad \theta_u\coloneqq u-\Pi_h u.
	$$
	The estimates of $\theta_{\bsig}$ and $\theta_u$ will be given in the subsequent subsection by proving the discrete inf-sup condition of \eqref{proj}. 

\subsection{Estimates of \texorpdfstring{$\theta_{\bsig}$}{Lg} and \texorpdfstring{$\theta_u$}{Lg}}
	Compared to the well-posedness in \cite[Theorem 3.1]{hu2021family}, this paper needs to deal with extra boundary conditions. Recall that $\mathcal{F}(\Omega)$ and $\mathcal{E}(\Omega)$ denote the sets of all faces and edges of $\Omega$, respectively. Define the space $\Sigma_0$ in $\mathbb{R}^3$ as
	$$	
	\Sigma_0 \coloneqq \{ \btau  \in H^1(\Omega;\mathbb{S}): \Pi_F(\boldsymbol{\tau}\boldsymbol{n}_F)=0, \, \forall F \in \mathcal{F}(\Omega)  \},
	$$
	and in $\mathbb{R}^2$ as
	$$	
	\Sigma_0 \coloneqq \{ \btau  \in H^1(\Omega;\mathbb{S}): \left. \boldsymbol{t}^{\topt}_e \btau \boldsymbol{n}_e\right|_{e} =0, \, \forall e \in \mathcal{E}(\Omega)  \}.
	$$
    For $\boldsymbol{v} \in L^2({\Omega;\mathbb{R}^d})$, define its symmetric gradient as $\symg \boldsymbol{v}   \coloneqq \frac{1}{2}(\nabla \boldsymbol{v} + \nabla \boldsymbol{v}^{\topt})$. The following lemma is crucial to the proof of the discrete inf-sup condition of \eqref{proj}.
	\begin{lemma}
		\label{thminfsup}
		For any $v_h \in V_h^0$, there exists a {matrix-valued function} $\btau \in \Sigma_0$ such that $\operatorname{divDiv} \btau = v_h$ and
		\begin{equation}
        \label{inequality45}
	    \operatorname{Div} \btau \in H^{1}_0(\Omega; \mathbb{R}^d), \quad
		\| \btau \|_{H^{1}} + \| \operatorname{Div} \btau \|_{H^{1}} \leq C\| v_h \|_{L^2},
		\end{equation}
		for some generic constant $C$.
	\end{lemma}
 
	The proof of Lemma \ref{thminfsup} requires the following two lemmas.
	\begin{lemma}[{Korn's inequality} {\cite[Theorem 3.2]{amrouche2006characterizations}}]
		\label{necas}
		There exists a constant $C>0$ such that
		$$
		\|\boldsymbol{v}\|_{L^2} \leq C\left(\| \boldsymbol{v}\|_{H^{-1}}+\| \symg \boldsymbol{v} \|_{H^{-1}}\right), \quad \forall \boldsymbol{v} \in L^2(\Omega; \mathbb{R}^3).
		$$
	\end{lemma}
	
	\begin{lemma}[{\cite[Theorem 6.3-4]{ciarlet2021mathematical}}]
		\label{Ciarlet}
		If $\boldsymbol{v} \in L^2(\Omega;\mathbb{R}^3)$ satisfies $\| \nabla_s \boldsymbol{v} \|_{H^{-1}}=0$, then $\boldsymbol{v}$ is a {rigid motion function} in $\mathbb{R}^3$, i.e., there exist two vectors $\boldsymbol{a}, \boldsymbol{b} \in \mathbb{R}^3$ such that $\boldsymbol{v}(\boldsymbol{x}) = \boldsymbol{a} + \boldsymbol{b} \times \boldsymbol{x}$.
	\end{lemma}
    \noindent
    Analogous arguments will show that Lemma \ref{necas}-\ref{Ciarlet} hold for $\Omega \subset \mathbb{R}^2$ as well. If $\boldsymbol{v} \in L^2(\Omega;\mathbb{R}^2)$ satisfies $\| \nabla_s \boldsymbol{v} \|_{H^{-1}}=0$, then there exist constants $a,b,c \in \mathbb{R}$ such that $\boldsymbol{v}(\boldsymbol{x}) = \begin{pmatrix}
    a-cy\\
    b+cx
    \end{pmatrix}$. 	
	\begin{proof}[Proof of Lemma \ref{thminfsup}]
		The proof for $\Omega \subset \mathbb{R}^2$ is similar to that for  $\Omega \subset \mathbb{R}^3$ which is detailed here. For any $v_h \in V_h^0$, there exists a $\boldsymbol{\psi} \in H^1_0(\Omega; \mathbb{R}^d)$ such that $\operatorname{div} \boldsymbol{\psi} = v_h$ \cite{girault2012finite}.
		It remains to prove the inf-sup condition
		\begin{equation}
			\label{eqinfsup}
			\sup_{\btau \in \Sigma_0} \frac{ (\operatorname{Div} \btau, \bphi)}{ \| \btau\|_{H^1} } \geq \beta \| \bphi\|_{L^2}, ~\forall \bphi \in L^2(\Omega;\mathbb{R}^d).
		\end{equation}
		Assume that \eqref{eqinfsup} is not valid. Then there would exist a sequence $\{ \bphi_n \}$ in $L^2$ such that $\| \bphi_n \|_{L^2} = 1$ and 
		\begin{equation}
			\label{eqinfsup1}
			\lim_{n \rightarrow \infty } \sup_{\btau \in \Sigma_0} \frac{ (\operatorname{Div} \btau, \bphi_n)}{ \| \btau\|_{H^1}} = 0.
		\end{equation}
		The first thing is to prove $\{ \bphi_n \}$ is a Cauchy sequence in $L^2$. Lemma \ref{necas} shows
		\begin{equation}
			\label{eqinfsupcauchy}
			\| \bphi_n - \bphi_m \|_{L^2} \leq  \| \bphi_n -\bphi_m \|_{H^{-1}} + \| \symg (\bphi_n - \bphi_m) \|_{H^{-1}}.
		\end{equation}
		The Rellich-Kondrachov compact embedding theorem \cite{adams2003sobolev} $H^{1}  \stackrel{c}{\hookrightarrow} L^2 \stackrel{c}{\hookrightarrow}  H^{-1}$ and $\| \bphi_n\|_{L^2} =1$ imply that $\{ \bphi_n \}$ is a Cauchy sequence in $H^{-1}$. 
		Additionally, 
		\begin{equation}
	\label{eqinfsup4}
	\begin{aligned}
		\| \symg( \bphi_n - \bphi_m)\|_{H^{-1}} & =  \sup_{\boldsymbol{\tau} \in H^{1}_0(\Omega;\mathbb{S})} \frac{ \langle \symg (\bphi_n -\bphi_m), \boldsymbol{\tau}\rangle_{H^{-1}\times H^1}}{\| \boldsymbol{\tau}\|_{H^1}}\\
		&=\sup_{\btau \in H_0^1(\Omega;\mathbb{S})} \frac{ (\operatorname{Div} \btau, \bphi_n-\bphi_m)}{ \| \btau\|_{H^1}}.
	\end{aligned}
    \end{equation}
		This, $H_0^1(\Omega;\mathbb{S})\subset \Sigma_0$, and \eqref{eqinfsup1} imply that $\{ \symg \bphi_n \}$ is a Cauchy sequence in $H^{-1}$.
		The previous arguments and \eqref{eqinfsupcauchy} show that there exists a $\bphi \in L^2(\Omega;\mathbb{R}^d)$ such that $\lim_{n \rightarrow \infty} \bphi_n = \bphi$ in $L^2$ {and} $\| \bphi\|_{L^2} = 1$. A triangle inequality leads to 
			\begin{equation}
				\label{triangle}
				\| \symg \bphi \|_{H^{-1}} \leq \| \symg (\bphi - \boldsymbol{\phi}_n )\|_{H^{-1}}+ \| \symg \boldsymbol{\phi_n} \|_{H^{-1}}.
			\end{equation}
			Proceeding as in \eqref{eqinfsup4} shows
			\begin{equation}
				\label{eqsimilar}
				\lim_{n \rightarrow \infty} \| \symg \boldsymbol{\phi_n} \|_{H^{-1}}=0.
			\end{equation}
			Since $\| \symg (\bphi - \boldsymbol{\phi}_n )\|_{H^{-1}} \leq \| \bphi - \boldsymbol{\phi}_n \|_{L^2}$, \eqref{triangle}-\eqref{eqsimilar} lead to $\| \symg \bphi \|_{H^{-1}}=0$. The combination with Lemma \ref{Ciarlet} implies that $\bphi$ is a {rigid motion function}. {This and} $\| \bphi\|_{L^2} = 1$ {imply} $\boldsymbol{n}^{\topt}\bphi \neq 0$ on $\partial \Omega$. {Consequently, one can choose} a $\btau_1 \in \Sigma_0$ such that $(\operatorname{Div} \btau_1, \bphi)=-(\btau_1\boldsymbol{n},\bphi)_{\partial\Omega} =-(\boldsymbol{n}^{\topt}\btau_1\boldsymbol{n},\boldsymbol{n}^\topt\bphi)_{\partial\Omega}> 0,$
			which contradicts \eqref{eqinfsup1}. This proves \eqref{eqinfsup}, and implies the existence of a desired $\btau$ with $ \operatorname{Div}{\btau}=\boldsymbol{\psi}$. This concludes the proof.
	\end{proof}

	Lemma \ref{thminfsup} gives rise to the following discrete inf-sup condition of \eqref{proj}.
 
     \begin{theorem}
     \label{infsup44}
     There exists a constant $\beta > 0$ such that 
     		\begin{equation*}
    			\label{proj2}
    			\sup_{\boldsymbol{\tau}_h\in\Sigma_{h}}\frac{( \operatorname{divDiv} \boldsymbol{\tau}_h,v_h)}{\|\boldsymbol{\tau}_h\|_{H(\operatorname{divDiv})}} \geq\beta \|v_h\|_{L^2}, \quad \forall v_h\in V_{h}^0. 
    		\end{equation*}
     \end{theorem}
     \begin{proof}
 	Given any $v_h \in V_{h}^0$, Lemma \ref{thminfsup} shows that there exists a $\btau \in \Sigma_0$ such that $\operatorname{divDiv} \boldsymbol{\tau} = v_h$ and $\btau$ satisfies \eqref{inequality45}. Define an interpolation operator $\widetilde{\Pi}_{h}$ in $\mathbb{R}^3$ 
		as in \cite[Subsec. 3.1]{hu2021family} with modifications
		$$
		\widetilde{\Pi}_{h} \btau(a) = 0, \quad \forall  a \in \mathcal{V}_h \backslash \mathcal{V}_h^i, $$
		$$
		\boldsymbol{t}_e^{\topt} ( \widetilde{\Pi}_{h} \btau) \boldsymbol{n}_{e,j} = 0, \boldsymbol{n}_{e,i}^{\topt} (\widetilde{\Pi}_{h} \btau) \boldsymbol{n}_{e,j} = 0,\quad \forall e \in \mathcal{E}_h \backslash \mathcal{E}_h^i, 1 \leq i, j \leq 2,
		$$
		and in $\mathbb{R}^2$ with modifications
		$$\widetilde{\Pi}_{h} \btau(a) = 0, \quad \forall  a \in \mathcal{V}_h \backslash \mathcal{V}_h^i.
		$$
		Then the modified interpolation operator $\widetilde{\Pi}_{h}$ preserves the zero boundary conditions in Lemma \ref{sufficientbd} so that $\widetilde{\Pi}_{h} \btau \in \Sigma_{h}$ and $\operatorname{divDiv} \widetilde{\Pi}_{h}\boldsymbol{\tau} = \operatorname{divDiv} \boldsymbol{\tau} = v_h$. 
		The boundedness   $$\|\widetilde{\Pi}_{h}\btau\|_{L^2} \leq C(\|\btau\|_{H^1}+\|\operatorname{Div}\btau\|_{H^1}),$$ and \eqref{inequality45} conclude the proof.
 \end{proof}
	
    The estimates for the projection error functions $\theta_{\bsig}$ and $\theta_u$ can be obtained immediately from Lemma \ref{leminfsup} and Theorem \ref{infsup44}.
	\begin{lemma}
		\label{lemerrorproj}
		For $k \geq 3$, it holds
		\begin{subequations}
			\label{projection}
			\begin{align}
				\left\| \tu \right\|_{L^2}+\left\|  \ts \right\|_{L^2} & \leq C h^{k-1}\|u\|_{H^{k+1}}, \label{projectiona}\\ 
				\left\|\frac{\partial \theta_u}{\partial t}\right\|_{L^2} +\left\|\frac{\partial \theta_{\bsig} }{\partial t}\right\|_{L^2} & \leq C h^{k-1}\left\|\frac{\partial u}{\partial t}\right\|_{H^{k+1}}, \label{projectionc}    \\
				{| \tu  |_{2,h}} & \leq C h^{ {k-3}}\|u\|_{H^{k+1}}. \label{projectionb}
			\end{align}
		\end{subequations}
	\end{lemma}
	\begin{proof}
    {
    The inf-sup condition in Theorem \ref{infsup44} gives the well-posedness of \eqref{proj}.
    Babu\v{s}ka Brezzi theory \cite{boffi2013mixed} and the interpolation error estimates lead to \eqref{projectiona}.
    The derivation of \eqref{proj} with respect to $t$ yields 
    $$
    \left\{
    \begin{aligned}
    (\operatorname{divDiv} (\Pi_h \frac{\partial \theta_{\boldsymbol{\sigma}}}{\partial t} - \frac{\partial \theta_{\boldsymbol{\sigma}}}{\partial t}), v_h) &= 0, \, \forall v_h \in V_h^0,    \\
    (\Pi_h \frac{\partial \theta_{\boldsymbol{\sigma}}}{\partial t} - \frac{\partial \theta_{\boldsymbol{\sigma}}}{\partial t}, \boldsymbol{\tau}_h) - (\operatorname{divDiv} \boldsymbol{\tau}_h, \Pi_h \frac{\partial \theta_u}{\partial t} - \frac{\partial \theta_u}{\partial t}) &=0, \, \forall \boldsymbol{\tau}_h \in \Sigma_h.
    \end{aligned} \right.
    $$
    Similar arguments prove \eqref{projectionc}.
    Lemma \ref{leminfsup} and Babu\v{s}ka Brezzi theory lead to \eqref{projectionb}. } 	
	\end{proof}

    
	\subsection{Error Estimates}
	
	This subsection derives the error estimates for the linearized fully discrete mixed scheme \eqref{ds}. For the convenience of the error analysis, assume that the solution $u(\boldsymbol{x},t)$ of \eqref{meq} satisfies the following regularities
	\begin{equation}
		\label{regularity}
		\begin{aligned}
			u \in L^{\infty}\left(0, T ; H^{k+1}\right), \,
			\frac{\partial u}{\partial t} \in L^{\infty}\left(0, T ; H^{k+1}\right),\,   \frac{\partial^2 u}{\partial t^2} \in L^{\infty}\left(0, T ; L^2\right). 
		\end{aligned}
	\end{equation}
	According to \eqref{deflinfty}, there exists a uniform bound {$M > 1$} independent of $n,h,\tau$  such that
	\begin{equation}
		\label{estimateM}
		\sup_{0 \leq t \leq T} \left\{  \|u\|_{L^{\infty}},\|u\|_{H^{k+1}},   \|\frac{\partial u}{\partial t} \|_{H^{k+1}},  \|\frac{\partial^2 u}{\partial t^2}\|_{L^{2}} \right\} \leq M.
	\end{equation}
	
	Let $(\bsig^n,u^n)$ denote the value of the exact solution $(\bsig,u)$ at the time step $n$. Let $\theta_{\bsig}^n, \theta_u^n$ denote the value of $\theta_{\bsig}, \theta_u$ at the time step $n$. With the estimates of $\theta_{\bsig}^n$ and $\theta_u^n$ in Lemma \ref{lemerrorproj}, it remains to estimate
	$$
	e_{\bsig}^n \coloneqq \Pi_h \bsig^n - \bsig_h^n,\quad e_u^n \coloneqq \Pi_h u^n - u_h^n, \quad \text { for } n=0,1,2, \ldots, N.
	$$
	The following lemma will be used in the estimates. 
	
	\begin{lemma}
		\label{lem}
		There exists a constant $C$ independent of $h$ and $\tau$ such that 
		\begin{equation}
			\label{l1}
			\left\|  v_h \right\|_{L^{\infty}}  \leq C |  v_h |_{2,h}, \quad   \forall v_h \in V_h.
		\end{equation}
	\end{lemma}
	
	\begin{proof}
    This follows with the arguments of \cite[Lemma 3.7]{brenner2017ac} on the enriching operators $E:V_h \rightarrow H^2(\Omega)$
    \cite{Carstensen2023, georgoulis2011posteriori, wangming2013}. Further details are omitted for brevity.
	\end{proof}
   
    To establish the error estimates for the fully discrete scheme \eqref{ds}, a key result is needed: $\| \eun \|_{L^\infty}$ is bounded. The proof of the following theorem uses mathematical induction to assert this result and then establishes the error estimates based on this result.
	
	\begin{theorem}
		\label{error}
		Under the regularity assumption \eqref{regularity}, there exist two positive constants $h_0$ and $\tau_0$ such that when $h<h_0$ and $\tau<\tau_0$, \eqref{ds} is uniquely solvable with $k \geq 3$ and the following error estimate holds
		\begin{equation}
			\label{digui}
			\max _{0 \leq n \leq N} \left(    \left\|e_u^n\right\|_{L^2}^2+\sum_{i=1}^n \tau\left\|e_{\bsig}^i\right\|_{L^2}^2  \right)     \leq C_*\left(\tau^2+h^{2 k-2}\right),
		\end{equation}
	where {$C_*=C_*(\varepsilon,M)$} is a positive constant independent of $n, h$ and $\tau$.
	\end{theorem}
	
	\begin{proof}
		To establish \eqref{digui}, the following estimate 
		\begin{equation}
			\label{digui2}
			\| \eun \|_{L^{\infty}} \leq M, \quad \forall0 \leq n \leq N,
		\end{equation}
		will be invoked iteratively. Both \eqref{digui} and \eqref{digui2} are proven for $n=0, \ldots, N$  by mathematical induction. At the initial time step, $$e_{\bsig}^0 = \Pi_h \bsig^0 - \bsig_h^0 = 0, \quad e_u^0 = \Pi_h u^0 - u_h^0 = 0 .$$ 
		This shows that \eqref{digui} and \eqref{digui2} hold for $n=0$. Assuming that \eqref{digui} and \eqref{digui2} hold for $n-1$, it remains to verify that both \eqref{digui} and \eqref{digui2} hold for $n$.
		
		\textbf{Step 1 derives the error equation of $\esn$ and $\eun$}. At time $t^n$, \eqref{mvar2} leads to
		\begin{equation*}
			\left\{   
			\begin{aligned}
				( \left. \frac{\partial u}{\partial t} \right|_{t^n} ,v_{h}) + (\operatorname{divDiv}  \bsig^{n}, v_{h}) &= 0,   &&\forall  v_{h} \in {V_h}, &  \\
				( \bsig^{n},  \btau_h )- (\operatorname{divDiv}\btau_{h}, u^{n})&= (  -\frac{1}{\varepsilon^{2}}f(u^{n})\I, \btau_{h} ),&&\forall \btau_{h} \in {\Sigma_h}.&
			\end{aligned}\right.
		\end{equation*}
		The combination with \eqref{proj} yields
		\begin{equation}
			\label{pds}
			\left\{  
			\begin{aligned}
				( D_{\tau} u^n ,v_{h}) + (\operatorname{divDiv} \Pi_h \bsig^{n}, v_{h}) &= ( R_1^n ,v_h), &&\forall  v_{h} \in  {V_h}, & \\
				(\Pi_h \bsig^{n},  \btau_h )- (\operatorname{divDiv}\btau_{h},\Pi_h u^{n})&= ( R_2^n ,\btau_h) + ( -\frac{1}{\varepsilon^{2}}f(u^{n-1})\I, \btau_{h}),  &&\forall \btau_{h} \in {\Sigma_h}&
			\end{aligned}\right.
		\end{equation} 
		with
		$$
		\begin{aligned}
			D_{\tau} u^n \coloneqq \frac{u^n-u^{n-1}}{\tau},\,
			R_1^n \coloneqq D_\tau u^n-\left. \frac{\partial u}{\partial t} \right|_{t^n},\,
			R_2^n \coloneqq \frac{1}{\varepsilon^{2}}f(u^{n-1})\I-\frac{1}{\varepsilon^{2}}f(u^{n})\I.
		\end{aligned}
		$$
		Subtracting the fully discrete scheme \eqref{ds} from \eqref{pds} shows
		\begin{equation}
			\label{eds}
			\left\{  
			\begin{aligned}
				( D_{\tau} e_u^n ,v_{h}) + (\operatorname{divDiv} \esn , v_{h}) &= -(D_{\tau}{\theta}_{u}^{n},v_{h}) + ( R_1^n ,v_h) ,&&\forall  v_{h} \in {V_h},&    \\ 
				( \esn ,  \btau_h )- (\operatorname{divDiv}\btau_{h},e_u^{n})&= ( R_2^n ,\btau_{h}) + (R_3^n , \btau_h),  &&\forall \btau_{h} \in {\Sigma_h}&
			\end{aligned}    \right.
		\end{equation}
		with $$R_3^n \coloneqq  \frac{1}{\varepsilon^{2}}f(u^{n-1}_h)\I -\frac{1}{\varepsilon^{2}}f(u^{n-1})\I.$$
		Taking $(\btau_h, v_h) = (\esn, \eun)$ in \eqref{eds} and summing up the results lead to the following error equation
		\begin{equation}
			\label{sum}
			\begin{aligned}
				\left(D_\tau e_u^n, e_u^n\right)+\left\| \esn \right\|_{L^2}^2 =& -\left(D_\tau \theta_u^n, e_u^n\right) + (R_1^n, e_u^n) + ( R_2^n, \esn ) +  ( R_3^n, \esn ). \\
			\end{aligned}
		\end{equation}
		
		\textbf{Step 2 estimates the terms on the right hand side of \eqref{sum}}.  For the first term of \eqref{sum}, the Taylor expansion with $\xi_{1} \in (t^{n-1},t^{n})$, \eqref{projectionc}, and \eqref{estimateM} lead to
		\begin{equation}
			\label{r2}
			\begin{aligned}
				-\left(D_\tau \theta_u^n, e_u^n\right)  &\leq \|D_\tau \theta_u^n\|_{L^2} \|e_{u}^{n}\|_{L^2} = \| \frac{\partial \theta_u}{\partial t}(\boldsymbol{x},\xi_{1})\|_{L^2} \left\|e_u^n\right\|_{L^2} \\
				&\leq C_1 h^{k-1}\|\frac{\partial u }{\partial t}(\boldsymbol{x},\xi_1)\|_{H^{k+1}} \|\eun\|_{L^2} \leq  C_1 Mh^{k-1}\|\eun\|_{L^2}
			\end{aligned}  
		\end{equation}
		with some constant $C_1$ from Lemma \ref{lemerrorproj}. 
		For the second term of \eqref{sum}, the Taylor expansion with $\xi_{2} \in (t^{n-1},t^{n})$ and \eqref{estimateM} give rise to 
		\begin{equation}
			\label{r1}
			\begin{aligned}
				(R_1^n, e_u^n) &= ( -\frac{1}{2} \frac{\partial^2 u}{\partial t^2}(\boldsymbol{x}, \xi_{2}) \tau , e_u^n ) \leq \frac{1}{2} \| \frac{\partial^2 u}{\partial t^2} (\boldsymbol{x}, \xi_{2}) \tau\|_{L^2} \|\eun\|_{L^2}\\
				&=\frac{1}{2}\tau\|\frac{\partial^2 u}{\partial t^2}(\boldsymbol{x}, \xi_{2})\|_{L^2} \|\eun\|_{L^2} \leq \frac{1}{2}M\tau \|\eun\|_{L^2}.
			\end{aligned}
		\end{equation}
		For the third term of \eqref{sum}, the Taylor expansion with $\xi_{3} \in (t^{n-1},t^{n})$ and \eqref{estimateM} show
		\begin{equation}
			\label{r3}
			\begin{aligned}
				&( R_2^n, \esn ) = \ep \left(f(u^{n-1}) - f(u^{n}), \operatorname{tr}  \esn \right) = \ep \left( \frac{\partial(u^3-u)}{\partial t} (\boldsymbol{x},\xi_{3}) \tau, \operatorname{tr}  \esn \right) \\
				&\quad \quad \quad \leq \ep \tau \|3u^{2}-1\|_{L^{\infty}} \|\frac{\partial u}{\partial t}(\boldsymbol{x},\xi_3)\|_{L^2} \|\operatorname{tr}\esn \|_{L^2} \leq \frac{3}{\varepsilon^{2}} M^3 \tau \left\| \operatorname{tr}\esn \right\|_{L^2}.
			\end{aligned}    
		\end{equation}
		Since \eqref{digui2} holds for $n-1$, it follows from Lemma \ref{lem}, \eqref{projectionb} and \eqref{estimateM} that
		\begin{equation*}
			\begin{aligned}
				\|u_{h}^{n-1}\|_{L^{\infty}} & \leq 
				\|e_{u}^{n-1}\|_{L^{\infty}} + \| \theta_{u}^{n-1}\|_{L^{\infty}} + \|u^{n-1}\|_{L^{\infty}} \\
				&\leq M+C_1 C_2h^{k-3}M+M = C_3M
			\end{aligned}
		\end{equation*}
		with some constant $C_2$ from Lemma \ref{lem} and $C_3 = C_1 C_2 h^{k-3}+2 $. This shows
		\begin{equation*}
			\begin{aligned}
				( R_3^n, \esn ) &= \ep \left( (u_h^{n-1}-u^{n-1}) \left( (u_h^{n-1})^2+(u^{n-1})^2+u_h^{n-1}u^{n-1}-1 \right),\operatorname{tr}\esn \right) \\
				& \leq \ep(C_3^2M^2+M^2+C_3M^2)    \|\theta_{u}^{n-1}+e_{u}^{n-1}\|_{L^2} \| \operatorname{tr}\esn  \|_{L^2}\\
				& \leq \ep  {C_4}M^2\|e_{u}^{n-1}\|_{L^2}\|\operatorname{tr}\esn\|_{L^2} + \ep {C_4}M^2\|\theta_{u}^{n-1}\|_{L^2}\|\operatorname{tr}\esn\|_{L^2}  \\
			\end{aligned}
		\end{equation*}
		with $C_4 = C_3 +C_3^2+1$. This and \eqref{projectiona} lead to the upper bound of the last term of \eqref{sum} by
		\begin{equation}
			\label{r44}
			\begin{aligned}
				( R_3^n, \esn ) \leq \ep {C_4}M^2\|e_{u}^{n-1}\|_{L^2}\|\operatorname{tr}\esn\|_{L^2} + \ep C_1  {C_4}M^3h^{k-1}\|\operatorname{tr}\esn\|_{L^2}.
			\end{aligned}
		\end{equation}
		
		\textbf{Step 3 proves the estimate \eqref{digui}}. The substitution of \eqref{r2}-\eqref{r44} into \eqref{sum} shows
		\begin{equation}
			\label{r}
			\begin{aligned}
				\left(D_\tau e_u^n, e_u^n\right)+ & \left\| \esn \right\|_{L^2}^2  \leq C_1 Mh^{k-1}\|\eun\|_{L^2}+\frac{1}{2}M\tau \|\eun\|_{L^2}+\frac{3}{\varepsilon^{2}} M^3 \tau \left\| \operatorname{tr}\esn \right\|_{L^2}\\
				&+ \ep {C_4}M^2\|e_{u}^{n-1}\|_{L^2}\|\operatorname{tr}\esn\|_{L^2} + \ep C_1  {C_4}M^3h^{k-1}\|\operatorname{tr}\esn\|_{L^2}.
			\end{aligned}
		\end{equation}
		Note that $\|\operatorname{tr}\esn\|_{L^2}^2 \leq d \|\esn\|_{L^2}^2$ with $d=2,3$. For any $\delta >0$,   this, \eqref{r} and Young's inequality imply
		\begin{equation}
			\label{rr}
			\begin{aligned}
				\left(D_\tau e_u^n, e_u^n\right)+\left\| \esn \right\|_{L^2}^2 & \leq \left(\frac{C_1^2 M^2}{2}+\frac{M^2}{8} + \frac{ 9M^6}{4\varepsilon^4\delta}+\frac{ C_1^2C_4^2
					M^6}{4\varepsilon^4\delta} \right)(\tau^2+h^{2k-2})  \\
				&~~+ \|\eun\|_{L^2}^2 + \frac{{C_4^2} M^4}{4\varepsilon^4\delta}\|e_{u}^{n-1}\|_{L^2}^2 + 3d\delta \|\esn\|_{L^2}^2.
			\end{aligned}
		\end{equation}
		Let 
		$${C}_5 =  \mathrm{max} \left\{ \frac{C_1^2 M^2}{2}+\frac{M^2}{8} +\frac{9M^6}{4\varepsilon^4\delta}+\frac{ C_1^2C_4^2
			M^6}{4\varepsilon^4\delta} , 1, \frac{ {C_4^2} M^4}{4\varepsilon^4\delta} \right\}. $$
		This and \eqref{rr} yield
		\begin{equation}
			\label{rrrr}
			\left(D_\tau e_u^n, e_u^n\right)+\left(1-3d\delta\right)\left\| \esn \right\|_{L^2}^2 \leq C_5\left(\tau^2+h^{2 k-2}\right) + C_5 \| \eun \|_{L^2}^2 + C_5 \| e_u^{n-1} \|_{L^2}^2.
		\end{equation}
		Taking $\delta < \frac{1}{3d}$ and multiplying $\frac{1}{1-3d\delta} > 1$ on both side of \eqref{rrrr}, the summation from $1$ to $n$ leads to
		$$
		\left\|e_u^n\right\|_{L^2}^2+\tau \sum_{i=1}^n\left\|e_{\bsig}^i \right\|_{L^2}^2 \leq \tau C_6 \sum_{i=1}^n\left\|e_u^i\right\|_{L^2}^2+\tau C_6 \sum_{i=1}^n\left(\tau^2+h^{2 k-2}\right)
		$$
		with $C_6=\frac{2C_5}{1-3d\delta}$. 
		If $\tau$ is sufficiently small such that $C_6 \tau \leq \frac{1}{2}$, then Lemma \ref{lemgron} shows
		$$
		\begin{aligned}
			\left\|e_u^n\right\|_{L^2}^2+\tau \sum_{i=1}^n\left\|e_{\bsig}^i\right\|_{L^2}^2 &\leq  \exp \left(\frac{T C_6}{1-C_6 \tau}\right)\left(  \tau C_6 \sum_{i=1}^n\left(\tau^2+h^{2 k-2}\right) \right) \\
			&\leq C_6 T \exp (2 T C_6)\left(\tau^2+h^{2 k-2}\right).
		\end{aligned}
		$$
		The choice of  $C_* \geq C_6 T \exp (2 T C_6)$ concludes that \eqref{digui} holds for $n$. 
		
		\textbf{Step 4 proves  $\| e_u^n \|_{L^{\infty}} \leq M.$} If $\tau \leq h^{k-1}$, then an inverse estimate with a constant $C_7$ and \eqref{digui} show
		\begin{equation}
			\label{compare1}
			\| \eun \|_{L^\infty} \leq C_7 h^{-\frac{d}{2}} \| \eun \|_{L^2}\leq C_7 h^{-\frac{d}{2}} \sqrt{ C_* (\tau^2 + h^{2k-2}) } \leq C_7 \sqrt{ 2C_*}h^{k-1-\frac{d}{2}} .
		\end{equation}
		If $\tau \geq h^{k-1}$, then Lemma \ref{lem}, Lemma \ref{leminfsup}, {and \eqref{eds}} lead to
		\begin{equation}
			\label{eun}
			\begin{aligned}
				\| e_u^n \|_{L^{\infty}} \leq& C_2 | \eun|_{2,h}  \leq \frac{C_2}{\beta} \sup_{{\btau_h} \in \Sigma_{h}} \frac{ \left(\eun, \operatorname{divDiv} {\btau_h}\right)}{ \| {\btau_h} \|_{L^2} } \\
				=&	\frac{C_2}{\beta}\sup_{{\btau_h} \in \Sigma_{h}} \left(\frac{ \left(\esn, {\btau_h}\right)}{ \| {\btau_h} \|_{L^2} } + \frac{ \left(f(u^n) - f(u_h^{n-1}), \operatorname{tr} {\btau_h}\right)}{ \varepsilon^2 \| {\btau_h} \|_{L^2} } \right)\\
				\leq & \frac{C_2}{\beta} \left\| \esn \right\|_{L^2} +  \frac{C_2 \sqrt{d}}{\varepsilon^2 \beta} \left\| f(u^n) - f(u_h^{n-1})\right\|_{L^2}.
			\end{aligned}
		\end{equation}
		For the first term on the right hand side of \eqref{eun}, \eqref{digui} shows 
		\begin{equation}
			\label{n-3}
			\| \esn \|_{L^2} \leq \tau^{-\frac{1}{2}} \sqrt{\tau \| \esn \|_{L^2}^2} \leq \tau^{-\frac{1}{2}} \sqrt{ C_* (\tau^2 + h^{2k-2}) } \leq \sqrt{ 2C_*}\tau^{\frac{1}{2}} .
		\end{equation}
		A triangle inequality, \eqref{r3}-\eqref{r44} plus \eqref{digui} imply
		\begin{equation}
			\label{n-1}
			\begin{aligned}
				\left\| f(u^n) - f(u_h^{n-1})\right\|_{L^2} &\leq 	\left\| f(u^n) - f(u^{n-1})\right\|_{L^2} + \left\| f(u^{n-1}) - f(u_h^{n-1})\right\|_{L^2} \\
				&\leq 3M^3\tau + {C_4}  M^2 \left\|e_u^{n-1}\right\|_{L^2} + C_1  {C_4} M^3h^{k-1}\\
				& \leq 3M^3\tau + {C_4}  M^2 \tau \sqrt{2C_*}+ C_1  {C_4} M^3h^{k-1}.
			\end{aligned}
		\end{equation}
		The substitution of \eqref{n-3}-\eqref{n-1} into \eqref{eun} gives
		\begin{equation}
			\label{compare2}
			\| \eun \|_{L^{\infty}} \leq \frac{C_2}{\beta}\sqrt{2C_*} \tau^{\frac{1}{2}} +\frac{C_2 \sqrt{d}}{\varepsilon^2 \beta} \left(  (3+C_1C_4) M^3 + C_4 M^2 \sqrt{2C_*} \right) \tau.
		\end{equation}
		Consequently, \eqref{compare1} and \eqref{compare2} show that there exist sufficiently small constants $\tau_0$ and $h_0$ such that $\| \eun \|_{L^{\infty}} \leq M$ holds for any $\tau \leq \tau_0$ and $h \leq h_0$. This concludes the proof.
	\end{proof}
	
	The theorem below follows immediately from Lemma \ref{lemerrorproj} and Theorem \ref{error}.
	
	\begin{theorem}
		\label{thmerror}
		Under the conditions in Theorem \ref{error}, \eqref{ds} is uniquely solvable with $k \geq 3$ and the following error estimate holds
		\begin{equation}
			\label{Error}
			\max _{0 \leq n \leq N}\left(\left\|u_h^n-u^n\right\|_{L^2}^2+\tau \sum_{i=1}^n\left\|\bsig_h^i-\bsig^i\right\|_{L^2}^2\right) \leq C \left(\tau^2+h^{2k -2 }\right),
		\end{equation}
	  where {$C=C(\varepsilon,M)$} is a positive constant independent of $n, h$ and $\tau$.
	\end{theorem}


	\section{Numerical Results}
	\label{sec:NumericalResults}
 
	This section presents several numerical examples in two and three dimensions to test the performance of the mixed finite element method. The discrete finite element spaces $\Sigma_h$ and $V_h$ with $k=3$ in Section \ref{sec4} are used.
		
	\captionsetup{font=small}
	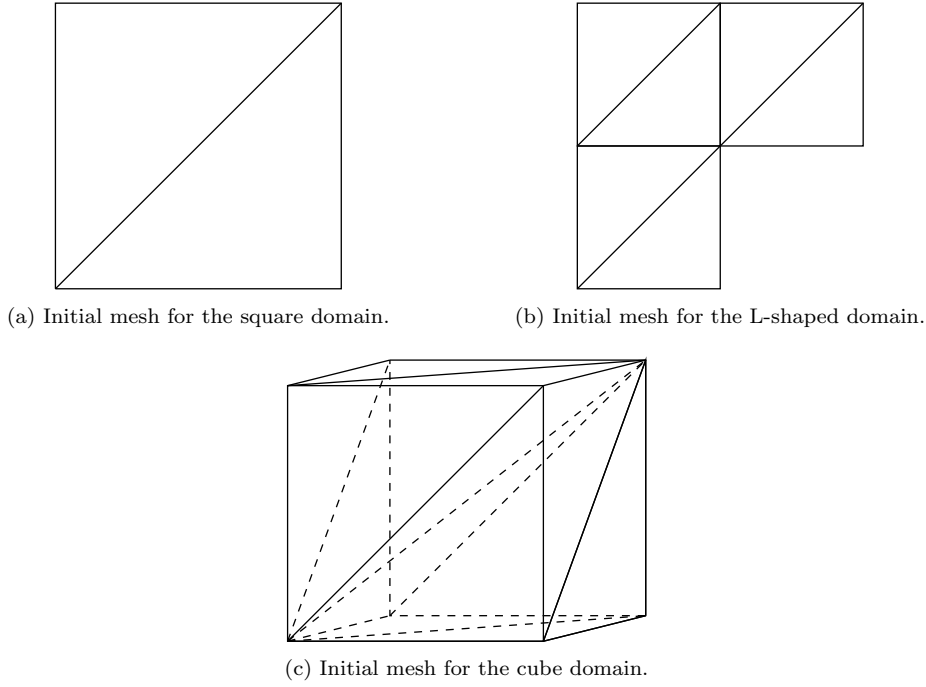
\begin{figure}[!ht]
		\centering
		\begin{minipage}[b]{0.45 \textwidth}
			\centering
			\begin{tikzpicture}[>=latex,line width=0.5pt,scale=0.95]
				\def  \a{4} 
				\draw (0,0) -- (\a,0) --(\a,\a) --(0,\a) --(0,0);
				\draw (0,0)--(\a,\a);
			\end{tikzpicture}
			\subcaption{Initial mesh for the square domain.}
			\label{fig:subplot1}
		\end{minipage}
		\hspace{0.05\textwidth}
		\begin{minipage}[b]{0.48\textwidth}
			\centering
			\begin{tikzpicture}[>=latex,line width=0.5pt,scale=0.95]
				\def  \a{2} 
				\draw (0,0) -- (\a,0) --(\a,\a) --(0,\a) --(0,0);
				\draw (0,0) -- (-\a,0) --(-\a,\a) --(0,\a) --(0,0);
				\draw (0,0) -- (0,-\a) --(-\a,-\a) --(-\a,0) --(0,0);
				\draw (0,0)--(\a,\a);
				\draw  (0,0)--(-\a,-\a);
				\draw (0,\a)--(-\a,0);
			\end{tikzpicture}
			\subcaption{Initial mesh for the L-shaped domain.}
			\label{fig:subplot2}
		\end{minipage}
		\vspace{1mm}
		\begin{minipage}[b]{0.5\textwidth}
			\centering
			\begin{tikzpicture}[line width=0.5pt,scale=1.7]
				\coordinate (A) at (0,0);
				\coordinate (B) at (2,0);
				\coordinate (C) at (2,2);
				\coordinate (D) at (0,2);
				\coordinate (A1) at (0.8,0.2);
				\coordinate (B1) at (2.8,0.2);
				\coordinate (C1) at (2.8,2.2);
				\coordinate (D1) at (0.8,2.2);
				\draw (A) -- (B) -- (C) -- (D) -- cycle;
				\draw (B1) -- (C1) -- (D1);
				\draw[dashed] (B1) -- (A1) -- (D1);
				\draw[dashed] (A) -- (A1);
				\draw (B) -- (B1);
				\draw (C) -- (C1);
				\draw (D) -- (D1);
				\draw[dashed] (A) -- (C1);
				\draw[dashed] (A) -- (B1);
				\draw[dashed] (A) -- (D1);
				\draw (A) -- (C);
				\draw (B) -- (C1);
				\draw (D) -- (C1);
				\draw[dashed] (A1) -- (C1);
				\draw (A) -- (B) -- (B1) -- (C1) -- (B);
			\end{tikzpicture}
			\subcaption{Initial mesh for the cube domain.}
			\label{fig:subplot3}
		\end{minipage}
		\caption{Initial meshes.}
		\label{fig:mainfigure}
	\end{figure}
	
	\subsection{Accuracy test in 2D}
	Let $\Omega=(0,1)^2$ be a square domain.  Consider an analytical solution of \eqref{teq} as
	\begin{equation*}
		\label{example1}
		u( \boldsymbol{x} ,t)= u( x,y ,t) = \exp (-t) \cos (\pi x) \cos (\pi y) .
	\end{equation*}
	The source term $g$ and boundary terms $g_1, g_2$ in \eqref{teq} are chosen correspondingly. The initial triangulation is shown in Figure \ref{fig:subplot1}. Each triangulation is refined into a half-sized triangulation uniformly, to get a higher level triangulation.
	In this example, take $\varepsilon=1$, and choose $\tau = 1 \mathrm{E}-5$ small enough and $\tau = h^2$, respectively.
	Table \ref{tab1} records the errors $\| \bsig^{N} - \bsig_{h}^{N}\|_{L^2}$, $\| u^{N} - u_{h}^{N}\|_{L^2}$, and the convergence rates for the discrete scheme \eqref{ds} with $k=3$. From Table \ref{tab1}, one can observe that both the convergence rates for $\bsig_h$ and $u_h$ equal $2$. This coincides with Theorem \ref{thmerror}.
 
	\subsection{Postprocessing}		
	This subsection provides a postprocessing technique to improve the convergence rates for $\boldsymbol{\sigma}_h$ and $u_h$. The postprocessing method is as follows:
     \begin{itemize}
     \item[\textnormal{(i)}] Obtain $(\boldsymbol{\sigma}_h^{n},u_h^{n})$ from \eqref{ds} for the time step $n \leq N$.
     \item[\textnormal{(ii)}] Find the postprocessed discrete solution $\tilde{u}_{h}^{N} \in \hat{V}_{h}$: for each $K \in \mathcal{T}_h$,
    \begin{equation*}
		\label{post}
		\left\{  
		\begin{aligned}
			(\tilde{u}_{h}^{N},v_{h})_K &=( {u}_{h}^{N},v_{h})_K, &&\forall  v_{h} \in V_h,&  \\
			(\nabla^2 \tilde{u}_{h}^{N}, \nabla^2 q_h)_K&= (\bsig_{h}^{N},  \nabla^2 q_h)_K +  (\ep f(u_h^{N})\I,  \nabla^2 q_h)_K,&& \forall q_{h} \in \hat{V}_h&
		\end{aligned}   \right.
	\end{equation*}
    with the piecewise polynomial space
	$$
	\hat{V}_{h}\coloneqq\left\{q_h \in L^2(\Omega):\left. q_h \right|_K \in P_{k+2}(K),\, \forall K \in \mathcal{T}_h\right\}.
	$$
	\item[\textnormal{(iii)}] Find the postprocessed solution $\tilde{\bsig}_{h}^{N} \in \Sigma_h$ such that  
	\begin{equation*}
		\label{schemepost}
		\left\{  
		\begin{aligned}
			(\operatorname{divDiv} \tilde{\bsig}_{h}^{N}, v_{h})&=\frac{1}{\tau} ({u}_{h}^{N-1} - {u}_{h}^{N},v_{h}), &&\forall  v_{h} \in {V_h},&  \\
			(\tilde{\bsig}_{h}^{N}, \btau_{h} )- (\operatorname{divDiv}\btau_{h},\hat{u}_{h}^{N})&= ( -\frac{1}{\varepsilon^{2}}f(\tilde{u}_{h}^{N})\I, \btau_{h} ),&& \forall \btau_{h} \in \Sigma_h.& 
		\end{aligned}   \right.
	\end{equation*}
    \end{itemize}
	Table \ref{tabpost} shows that the convergence rates for $\tilde{\bsig}_h$ and $\tilde{u}_h$ are improved to 4, while the convergence rate for $u_h$ equals 2.

 	\begin{table}[!ht]
		\centering
		\caption{Errors on the square domain.}
		\label{tab1}
		\setlength{\tabcolsep}{4.5mm}{
			\begin{tabular}{c c c c  c}
				\toprule
				Mesh &$\| \bsig^{N} - \bsig_{h}^{N}\|_{L^2}$ & Rate & $\| u^{N} - u_{h}^{N}\|_{L^2}$  & Rate \\ 
				\hline
				& \multicolumn{2}{c}{$\varepsilon=1,\tau = 1 \mathrm{E}-5, T= 2\mathrm{E}-4$} &&\\
				\hline
				1 &   3.26E+00&	-	&2.69E-01&	- \\ 
				2 &   2.61E-01&	3.64	&7.37E-02&	1.87 \\ 
				3 &  2.62E-02&	3.32	&1.95E-02&	1.92 \\ 
				4 & 4.46E-03&	2.56	&4.95E-03&	1.98 \\ 
				5 &  1.07E-03&	2.05	&1.24E-03&	1.99 \\
				6 & 2.68E-04&	2.00	&3.11E-04&	2.00 \\
				\hline
				& \multicolumn{2}{c}{$\varepsilon=1,\tau = h^2, T= 1$} &&\\
				\hline
				1 &   1.18E+00&	-	    &1.06E-01&	- \\ 
				2 &   1.02E-01&	3.54	&2.71E-02&	1.97 \\ 
				3 &  1.23E-02&	3.04	&7.17E-03&	1.92 \\ 
				4 & 2.69E-03&	2.19	&1.82E-03&	1.98 \\ 
				5 &  6.71E-04&	2.00	&4.58E-04&	1.99 \\
				\bottomrule
		\end{tabular}}
	\end{table}

 	\begin{table}[!ht]
		\centering
		\caption{Errors on the square domain with postprocessing.}
		\label{tabpost}
		\setlength{\tabcolsep}{0.8mm}{
			\begin{tabular}{c c  c c  c c  c c}
				\toprule
				$\| \bsig^{N} - \bsig_{h}^{N}\|_{L^2}$ & Rate &$\| \bsig^{N} - \tilde{\bsig}_{h}^{N}\|_{L^2}$ & Rate & $\| u^{N} - u_{h}^{N}\|_{L^2}$  & Rate &  $\| u^{N} - \tilde{u}_{h}^{N}\|_{L^2}$ & Rate\\ 
				\hline
				& \multicolumn{4}{c}{$\varepsilon=1,\tau = 1 \mathrm{E}-6, T= 4\mathrm{E}-5$} &&\\
				\hline
				3.26E+00  &  - & 3.26E+00&   	-	&2.69E-01  &   -  & 1.32E-01   & -\\ 
				2.60E-01  &  3.65 & 2.56E-01&	3.67  &7.37E-02 & 1.87 & 9.86E-03 & 3.74\\ 
				2.64E-02  &  3.30 & 2.10E-02&	3.61 &1.95E-02  & 1.92  & 7.47E-04 & 3.72\\ 
			    4.46E-03  &  2.56 & 1.37E-03&  3.93  &4.95E-03 & 1.98 & 4.83E-05 & 3.95\\ 
				\hline
				& \multicolumn{4}{c}{$\varepsilon=1,\tau = h^4, T= 1$} & & \\
				\hline
				1.18E+00  &   - & 1.17E+00&   	-	&9.94E-02  &   -  & 4.83E-02   & -\\ 
				9.74E-02  &   3.60 & 9.58E-02&	3.61  &2.72E-02 & 1.87 & 5.15E-03 & 3.23\\ 
				8.51E-03  &  3.52 & 6.94E-03&	3.79  &7.17E-03  & 1.92  & 2.65E-04 & 4.28\\ 
				1.38E-03  &  2.62 & 4.48E-04&  3.95  & 1.82E-03 & 1.98 & 1.71E-05 & 3.95\\ 
				\bottomrule
		\end{tabular}}
	\end{table}
    
	\subsection{Accuracy test on the L-shaped domain}
    \label{secLdomain}
	Let $\Omega=(-1,1)^2 \backslash([0,1] \times[-1,0])$ be an L-shaped domain. Consider an analytical solution of \eqref{teq} as
	\begin{equation*}
		\label{example3}
		u(\boldsymbol{x}, t) = u(r, \theta, t)= \exp (-t) r^{\frac{4}{3}}\cos( \frac{2}{3}\theta).
	\end{equation*}
	The function $g$ and boundary conditions $g_1$ and $g_2$ are chosen correspondingly. The initial triangulation is shown in Figure \ref{fig:subplot2}. In this example, take $\varepsilon=1$, and choose $\tau = 1 \mathrm{E}-5$ small enough and $\tau = h^2$, respectively.
	Table \ref{tabLshape} records the errors $\| \bsig^{N} - \bsig_{h}^{N}\|_{L^2}$, $\| u^{N} - u_{h}^{N}\|_{L^2}$, and the convergence rates for the discrete scheme \eqref{ds} with $k=3$. From Table \ref{tabLshape}, the convergence can still be observed on the L-shaped domain. The convergence rates are degenerate because the solution possesses singularity at the origin. The rate equals $0.33$ for $\bsig_h$ and $0.66$ for $u_h$.

	\begin{table}[!ht]
		\centering
		\caption{Errors on the L-shaped domain.}
		\label{tabLshape}
		\setlength{\tabcolsep}{4.5mm}{
			\begin{tabular}{c c c c  c}
				\toprule
				Mesh &$\| \bsig^{N} - \bsig_{h}^{N}\|_{L^2}$ & Rate & $\| u^{N} - u_{h}^{N}\|_{L^2}$  & Rate \\ 
				\hline
				& \multicolumn{2}{c}{$\varepsilon=1,\tau = 1 \mathrm{E}-5, T= 2\mathrm{E}-4$} &&\\
				\hline
				1 &   2.81E+00&	-	       &5.53E-01&	- \\ 
				2 &   1.96E+00&	0.52	&3.78E-01&	0.55 \\ 
				3 &  1.44E+00&	0.45	&2.52E-01&	0.59 \\ 
				4 &  1.09E+00&	0.40	&1.62E-01&	0.64 \\ 
				5 &  8.44E-01&	0.37	&1.00E-01&	0.69 \\ 
				\hline
				& \multicolumn{2}{c}{$\varepsilon=1,\tau = h^2, T= 1$} &&\\
				\hline
				1 &         7.05E-01  &   	-	   &2.64E-01  &	   -    \\ 
				2 &      5.89E-01&	  0.26	 &1.85E-01  &	0.51 \\ 
				3 &      4.71E-01&	   0.32	  &1.20E-01  &	 0.63 \\ 
				4 &    3.75E-01&	 0.33  	&7.50E-02 &   0.68 \\ 
				5 &  2.98E-01&	0.33   &4.56E-02 &	 0.72 \\
				\bottomrule
		\end{tabular}}
	\end{table}
    
	\subsection{Accuracy test in 3D}
	Let $\Omega=(0,1)^3$ be a cubic domain.  Consider an analytical solution of \eqref{teq} as
	\begin{equation*}
		\label{example2}
		u(\boldsymbol{x},t) = u(x, y, z,t)=\exp (-t) \cos (\pi x) \cos (\pi y) \cos(\pi z).
	\end{equation*}
	The source term $g$ and boundary terms $g_1, g_2$ in \eqref{teq} are chosen correspondingly. The initial triangulation is shown in Figure \ref{fig:subplot3}. In this example, take $\varepsilon=1$, and choose $\tau = 1 \mathrm{E}-5$ small enough and $\tau = h^2$, respectively.
	Table \ref{tab3D} records the errors $\| \bsig^{N} - \bsig_{h}^{N}\|_{L^2}$, $\| u^{N} - u_{h}^{N}\|_{L^2}$, and the convergence rates for the discrete scheme \eqref{ds} with $k=3$. Table \ref{tab3D} shows that the convergence rate for $u_h$ equals $2$. This coincides with Theorem \ref{thmerror}. 
	
	\begin{table}[!ht]
		\centering
		\caption{Errors on the cube domain.}
		\label{tab3D}
		\setlength{\tabcolsep}{4.5mm}{
			\begin{tabular}{c c c c  c}
				\toprule
				Mesh &$\| \bsig^{N} - \bsig_{h}^{N}\|_{L^2}$ & Rate & $\| u^{N} - u_{h}^{N}\|_{L^2}$  & Rate \\ 
				\hline
				& \multicolumn{2}{c}{$\varepsilon=1,\tau = 1 \mathrm{E}-5, T= 2\mathrm{E}-4$} &&\\
				\hline
				1 &  4.89E+00&	-	&2.45E-01&	- \\ 
				2 &   4.87E-01&	3.33	&6.41E-02&	1.94 \\ 
				3 &  4.07E-02&	3.58	&1.73E-02&	1.89 \\ 
				4 &  4.91E-03&	3.05	&4.42E-03&	1.97 \\ 
				\hline
				& \multicolumn{2}{c}{$\varepsilon=1,\tau = h^2, T= 1$} &&\\
				\hline
				1 &   1.78E+00&	-	    &9.13E-02&	- \\ 
				2 &   1.78E-01&	3.33	&2.36E-02&	1.95 \\ 
				3 &  1.50E-02&	3.57	&6.35E-03&	1.89 \\ 
				\bottomrule
		\end{tabular}}
	\end{table}

	\subsection{Coalescence of two drops}
	Consider the coalescence of two material drops governed by the Cahn-Hilliard equation \eqref{teq} {with an additional parameter $m=$1E-2 as follows
    $$
    \frac{\partial u}{\partial t} -m^2\Delta\left(-\Delta u+\frac{1}{\varepsilon^2} f(u)\right) = g(\boldsymbol{x},t) \text{ in} ~ \Omega \left. \times(0, T\right].
    $$} 
    Assume that, at time $t = 0$, the first material occupies two circular regions that are right next to each other and the second material fills the rest of the domain. The two regions of the first material then coalesce with each other to form a single drop under the Cahn-Hilliard dynamics. The initial distribution for the materials from \cite{Yang2019} reads
	\begin{equation*}
		u_0(\boldsymbol{x})=1-\tanh \frac{\left|\boldsymbol{x}-\boldsymbol{x}_0\right|-R_0}{\sqrt{2} \varepsilon}-\tanh \frac{\left|\boldsymbol{x}-\boldsymbol{x}_1\right|-R_0}{\sqrt{2} \varepsilon},
	\end{equation*}
	where $\varepsilon=0.01$ is the characteristic inter-facial thickness, $\boldsymbol{x}_0=(0.3,0.5)$ and $\boldsymbol{x}_1=(0.7,0.5)$ are the centers of the circular regions for the first material, and $R_0=0.19$ is the radius of these circles. Set $g=0, g_1=0, g_2=0$ of \eqref{teq} for simulation of natural phenomena. The process of coalescence of the two drops is demonstrated in Figure \ref{figdrops} with a temporal sequence of snapshots of the
	interfaces between the materials (visualized by the contour level $u= 0$).

	\begin{figure}[!ht]
		\centering
		
		\begin{subfigure}[b]{0.48\textwidth}
			\includegraphics[width=\textwidth]{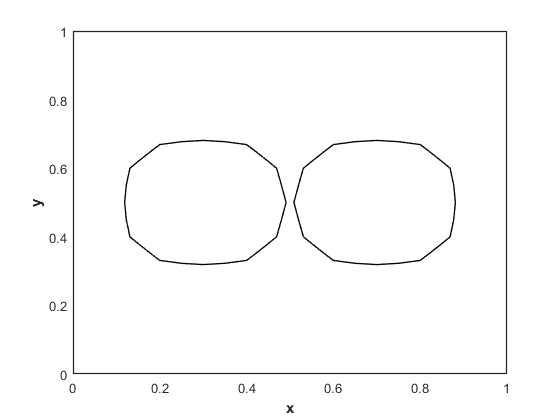}
			\caption{t=0}
		\end{subfigure}
		\quad
		\begin{subfigure}[b]{0.48\textwidth}
			\includegraphics[width=\textwidth]{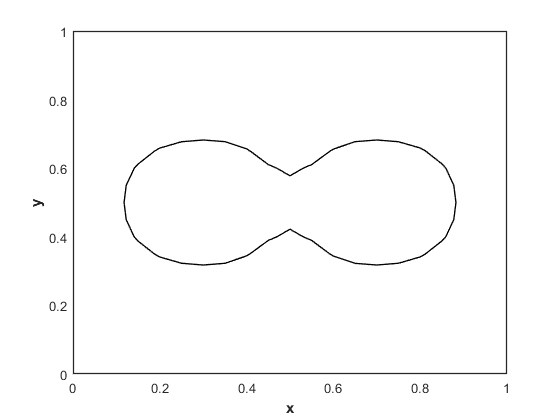}
			\caption{t=1}
		\end{subfigure}
		
		\begin{subfigure}[b]{0.48\textwidth}
			\includegraphics[width=\textwidth]{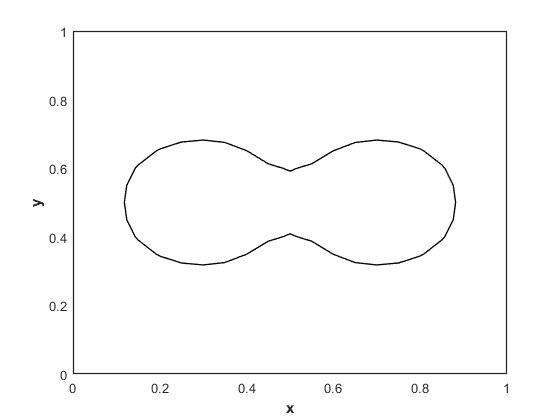}
			\caption{t=2}
		\end{subfigure}
		\quad
		\begin{subfigure}[b]{0.48\textwidth}
			\includegraphics[width=\textwidth]{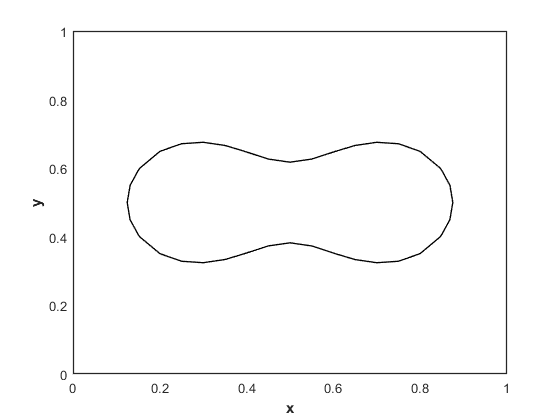}
			\caption{t=10}
		\end{subfigure}
		
		\caption{Temporal sequence of snapshots for the coalescence of two drops. Results are obtained with $\tau = 1E-2$ and $h=\sqrt{2}/20$.}
		\label{figdrops}
	\end{figure}

    \section{Concluding remarks}
    \label{sec:concluding}
    
    This paper developed a new mixed finite method for the Cahn-Hilliard equation. The well-posedness of the mixed formulation was provided by proving its equivalence with the primal formulation. The error estimates of the linearized fully discrete scheme was established by mathematical induction. The boundedness of $\|u_h\|_{L^\infty}$ was proved. Some postprocessing technique was given to improve the convergence rates which was verified by numerical tests. The theoretical result for the postprocessing could be obtained by following similar arguments as in \cite[Theorem 2.1]{wang2016postprocessing}, which analyzed the postprocessing for solving Cahn-Hilliard equations by the Ciarlet-Raviart method. 
    
    Since the Cahn-Hilliard equation describes the dissipation of the Cahn-Hilliard free energy \cite{zhang2010nonconforming} in a conservative system, it is desirable for a numerical scheme to preserve this energy property. One future work will investigate the discrete energy of the current mixed finite element method. Besides, this paper has only discussed the case with the interface parameter $\varepsilon=1$. However, a small $\varepsilon \ll 1$ is important in applications and brings difficulties in computation. The numerical schemes when $\varepsilon$ approaches zero will be discussed in the future work.

\section{Declaration of competing interest}
The authors declare no competing interests.
\section{Data availability}
No data was used for the research described in the article.

 \bibliographystyle{elsarticle-num} 
 \bibliography{references}

\end{document}